\documentclass[11pt]{amsart}
\usepackage{amscd}
\usepackage[arrow,matrix]{xy}
\input xypic
\usepackage{graphicx, tikz}
\usepackage{hyperref}
\usepackage{soul}
\usepackage{comment}
\usepackage{amsmath}
\usepackage{mathtools}
\usepackage{amsmath, latexsym, amssymb}
\numberwithin{equation}{section}
\theoremstyle{plain}
\newtheorem{lemma}{Lemma}[section]
\newtheorem{proposition}[lemma]{Proposition}
\newtheorem{theorem}[lemma]{Theorem}
\newtheorem{corollary}[lemma]{Corollary}
\newtheorem{problem}[lemma]{Problem}

\theoremstyle{definition}
\newtheorem{definition}[lemma]{Definition}
\newtheorem*{definition*}{Definition}
\newtheorem{remark}[lemma]{Remark}
\newtheorem{example}[lemma]{Example}
\usepackage{color}
\usepackage{cancel}

\definecolor{brown}{RGB}{150,100,0}

\definecolor{purple}{RGB}{150,0,100}
\definecolor{grey}{RGB}{128,128,128}

\newcommand{\R}{{\mathbb R}}

\newcommand{\E}{{\mathbb E}}

\newcommand{\N}{{\mathbb N}}

\newcommand{\Aa}{{\mathcal A}}
\newcommand{\Bb}{{\mathcal B}}
\newcommand{\Cc}{{\mathcal C}} 
\newcommand{\Dd}{{\mathcal D}}

\newcommand{\Ff}{{\mathcal F}}
\newcommand{\Gg}{{\mathcal G}} 
\newcommand{\Hh}{{\mathcal H}}

\newcommand{\Mm}{{\mathcal M}} 

\newcommand{\Pp}{{\mathcal P}}

\newcommand{\Ss}{{\mathcal S}}

\newcommand{\Xx}{{\mathcal X}}
\newcommand{\Yy}{{\mathcal Y}}
\newcommand{\Zz}{{\mathcal Z}}

\mathchardef\mhyp="2D

\newcommand{\Om}{{\Omega}}
\newcommand{\om}{{\omega}}
\newcommand{\eps}{{\varepsilon}}

\newcommand{\Meas}{{\bf Meas}}

\def\NABLA#1{{\mathop{\nabla\kern-.5ex\lower1ex\hbox{$#1$}}}}
\def\Nabla#1{\nabla\kern-.5ex{}_#1}

\newcommand{\INTO}{\hookrightarrow}

\renewcommand{\:}{\colon}

\mathchardef\mhyp="2D

\DeclareMathOperator{\Probm}{Probm}

\DeclareMathOperator{\Id}{Id}

\newcommand{\pb}{{\mathbf p}}
\newcommand{\qb}{{\mathbf q}}

\begin{document}
	\title[Bayesian  batch learning  equals  online learning]{Batch learning   equals  online learning in Bayesian supervised learning}
	\author[H. V. L\^e]{H\^ong V\^an L\^e}	
	\address{Institute of Mathematics of the Czech Academy of Sciences, Zitna 25, 11567 Praha 1, Czech Republic}
	
	\email{hvle@math.cas.cz}
	\date{\today}
	
	\begin{abstract}  In this  paper  we study   Bayesian supervised  learning models proposed  by L\^e in \cite{Le2025}. Using functoriality of probabilistic morphisms, we prove that sequential and batch Bayesian inversions coincide in supervised learning models with conditionally independent (possibly non-i.i.d.) data \cite{Le2025}. This equivalence holds without domination or discreteness assumptions on sampling operators. We derive a recursive formula for posterior predictive distributions, which reduces to the Kalman filter in Gaussian process regression. For Souslin label spaces $\Yy$ and arbitrary input sets $\Xx$, we characterize probability measures on $\Pp(\Yy)^\Xx$ via projective systems, generalizing Orbanz \cite{Orbanz2011}. We revisit MacEachern’s dependent Dirichlet processes (DDP) \cite{MacEachern2000} using copula-based constructions \cite{BJQ2012} and show how to compute posterior predictive distributions in universal Bayesian supervised models with DDP priors.
	\end{abstract}

	\maketitle
	
	\section{Introduction}
	For a measurable space $\Xx$,  we  denote  by  $\Sigma_\Xx$ the $\sigma$-algebra of $\Xx$, and by $ \Pp (\Xx)$ the  space  of  all probability measures  on $\Xx$.  If $(\Xx,\tau)$ is a  topological space, we  consider the Borel $\sigma$-algebra $\Bb(\tau)$, denoted also by $\Bb (\Xx)$, on $\Xx$, unless otherwise  stated.
	
	In \cite{Le2025} we considered the following problem.

	\begin{problem}[{\bf Supervised  Bayesian Inference  (SBI) Problem}]\label{prob:transuper} Let $\Xx$  be  an input  space and  $\Yy$    a measurable label  space.  Given  training  data
		$S_n: =\big ((x_1, y_1), \ldots , (x_n, y_n)\big) \in  (\Xx \times \Yy)^n$  and   new  test data $T_m :  = (t_1, \ldots, t_m)\in \Xx^m$,  estimate  the  {\it predictive   probability  measure}   $\Pp_{ T_m|S_n} \in \Pp (\Yy^m)$  that  governs  the joint distribution of  the   $m$-tuple  $ \big(y_1', \ldots, y_m'\big) \in \Yy^m$  where $y_i'$  is the label of $t_i$.   
	\end{problem}

	If $\Xx$  consists  of a single point,   Problem SBI  is equivalent  to   the   fundamental problem of  probability estimation in classical mathematical statistics. If $m =1$  and $\Yy = \R^n$, under the  assumption that  the   distribution of the label  $y$ of $x$ is governed by  a corrupted  measurement of $f(x)$  for some  unknown function $f: \Xx \to \R^n$,   Problem SBI  is the    regression problem in classical statistics.
	
	In  \cite{Le2025}, utilizing  a categorical approach and  stochastic processes  taking values in $\Pp (\Yy)$ with index set $\Xx$, we proposed
	a  Bayesian solution of   Problem  \ref{prob:transuper} encompassing  classical solutions of   probability   and  regression estimation problems that use Bayesian
	inversions.  Our  Bayesian   modeling  (Definition \ref{def:fposteriord})   of   Problem \ref{prob:transuper}   works under  the assumption of  conditionally independent  (possibly  not identically distributed)  data $y \in \Yy$, which  encompasses  the classical  Bayesian   modeling  of   conditionally  i.i.d.  data  $y \in \Yy$, assuming $\#(\Xx)=1$, see also Remark  \ref{rem:uniqueness}(2+3). The classical Bayesian modeling is based  on de Finetti's theorem   on  exchangeable  data and its generalizations.  
	
	In this paper,   we    prove that  batch learning  equals  online  learning  in Bayesian  supervised learning  (Theorems  \ref{thm:binv},  \ref{thm:posspred}, \ref{thm:posspredr}). 
	
	The  question  of whether batch learning  equals online learning  in  Bayesian  learning  has a notable  history and  importance  in mathematical  statistics and machine learning.
	The formal study of updating statistical conclusions one observation at a time (online learning) is known as sequential analysis. The mathematical groundwork for sequential analysis   was  laid by   Abraham Wald \cite{Wald1947}. The explicit formalization of sequential updating in a Bayesian context for estimation problems came to prominence with the development of state-space models and filtering theory \cite{SS2023}.
	Rudolf E. K\'alm\'an is arguably the most important figure in the practical application of this principle. The Kalman Filter, introduced around 1960, is a perfect example of online Bayesian learning. It uses the posterior from the previous time step as the prior for the current time step to recursively estimate the state of a linear dynamic system. While the Kalman filter is a specific algorithm for Gaussian models, its conceptual basis is precisely the equivalence of batch and online updating. The broader theoretical treatment of this idea in Bayesian statistics is often attributed to Dennis V. Lindley  and Adrian F.M. Smith. Their work in the 1970s on Bayesian hierarchical models and the structure of Bayesian inference helped formalize and popularize these recursive computational structures. For instance, their 1972 paper \cite{LS1972}  is a landmark in this area.
	 For   Bayesian models with conjugate  priors, the online approach is simply a recursive way of performing the same computation as the batch approach, breaking it down into  smaller, manageable steps without any loss of information or change in the final inference. To the  best of  our knowledge,  until  now, the most   general  available theorem  stating that batch  Bayesian learning equals   online    Bayesian learning    relies on the  assumption    of the  classical Bayes'  theorem, which  assumes  either discrete  data  or   sampling operators that are dominated  Markov kernels, see, e.g.,  \cite[Section 3.3]{SS2023}. 
	 
	   Our   approach  is based  on   properties of   compositions of  Markov kernels and their graphs, called  functoriality of probabilistic morphisms, which we studied in \cite{Le2025} and
	  our characterization  of Bayesian inversions  as  solutions of   operator  equation  for the graph of a Markov kernel  \eqref{eq:bayesinv0}. Prior results on recursive Bayesian updating require additional structure such as dominated kernels or conjugate priors \cite{SS2023}. 
	  In contrast, we establish the equivalence of batch and online learning using only 
	  conditional independence, functoriality of probabilistic morphisms,  and the 
	  operator equation characterization of Bayesian inversion 
	  developed in \cite{Le2025}, 
	  combined with the functorial structure of sequential inversion 
	  established in Theorem~\ref{thm:binv} of the present work.  As a result,   we can prove the equivalence without the i.i.d. assumption, holding for the more general case of conditionally independent (possibly not identically distributed) data that our model (Definition \ref{def:fposteriord}) addresses.
	  
	The  equivalence    is vital  because  online learning  offers advantages  such as   higher computational  efficiency  and adaptability, see, e.g., Example \ref{ex:Gauss}.

	Our paper is organized as follows.  In  Section \ref{sec:pre}, we  recall   the concept of  probabilistic  morphisms,      their  useful properties,   and    Bayesian  learning models  for supervised  learning (Definition \ref{def:fposteriord}).
	In   Section \ref{sec:bayin}, using  a projective system,  we   derive   a formula  for  Bayesian  inversions of  a universal supervised learning  model $(\Pp (\Yy)^\Xx, \mu, \Id_{\Pp (\Yy) ^\Xx}, \Pp (\Yy)^\Xx)$ if  $\Xx$ is a finite  set (Theorem \ref{thm:projbi}). 
 We illustrate  Theorems \ref{thm:binv} and \ref{thm:projbi} by computing  the  posterior distributions of Dirichlet processes  (Example \ref{ex:dirichleti}).
 	In Section \ref{sec:pred},  we prove  recursive formulas  for  posterior  predictive distributions (Theorems \ref{thm:posspred}, \ref{thm:posspredr}) and  illustrate Theorem \ref{thm:posspredr} with Gaussian process regressions (Example \ref{ex:Gauss}).  Section \ref{sec:bsuperv},   assuming that  $\Yy$ is a Souslin space  and $\Xx$ is an arbitrary set,      characterizes probability measures  on  $\Pp (\Yy) ^\Xx$    via  a  projective  system,  defined  by  finite subsets  in $\Xx$ and   a countable  generating     algebra  of   the $\sigma$-algebra of $\Yy$  (Theorem \ref{thm:uniprior}).  In  Section \ref{sec:BDDP}, we illustrate
	Theorem  \ref{thm:uniprior}  with   MacEachern's Dependent  Dirichlet Processes  (DDP) priors (Theorem \ref{thm:ddp}) and indicate how to compute posterior  predictive distributions  of    universal  Bayesian  supervised  learning models with  DDP priors.    In the last Section \ref{sec:fin},   we   discuss  our results   and  the concept of predictive consistency  in  Bayesian   supervised learning.
	\
	
\section{Preliminaries}\label{sec:pre}

	- For  a   measurable  space $\Xx$,  we  denote by $\Sigma _w   $  the smallest $\sigma$-algebra  on $\Pp (\Xx)$ such that  for  any $A \in \Sigma _\Xx$  the function $e_A: \Pp (\Xx) \to \R, \mu \mapsto    \mu(A), $ is  measurable.    In our paper,  we always  consider  $\Pp (\Xx)$  as  a measurable   space  with    the   $\sigma$-algebra  $\Sigma_w$, unless  otherwise stated.
	
	-  For a measurable space $\Xx$, we denote  by $\Ff_b (\Xx)$  and  $\Ff_s (\Xx)$   the  space of measurable bounded functions and the space of all step functions on $\Xx$, respectively.

	- A Markov  kernel $T: \Xx \times  \Sigma_\Yy \to [0,1]$    is  uniquely defined   by    the measurable map  $ \overline T:  \Xx \to  \Pp (\Yy)$  such that  $\overline T (x) (A) = T (x, A)$ for all $x\in \Xx, A \in \Sigma _\Yy$.   We  shall   also use notations $T ( A|x):= T (x, A) $ and $\overline T(A|x) : = \overline T (x)(A)$.  
	
	- A    {\it probabilistic  morphism}  $T: \Xx \leadsto \Yy$  is an  arrow assigned to a  measurable mapping, denoted  by $\overline T$,  from $ \Xx$ to  $\Pp (\Yy)$.  We say  that  $T$ is generated  by $\overline  T$.  For a measurable mapping $T: \Xx \to \Pp (\Yy)$ we  denote  by $\underline T: \Xx \leadsto  \Yy$ the generated  probabilistic morphism.
	
	-  For   probabilistic  morphisms  $T_{\Yy|\Xx} : \Xx \leadsto \Yy$ and  $T_{\Zz|\Yy}: \Yy \leadsto  \Zz$  their    composition is the  probabilistic  morphism
	$$  T_{\Zz|\Xx} : = T_{\Zz |\Yy} \circ  T _{\Yy|\Xx}: \Xx \leadsto \Zz  $$
	$$	(T_{\Zz|\Yy}\circ T_{\Yy|\Xx}) (x, C): = \int_\Yy T_{\Zz|\Yy} (y ,C) T_{\Yy| \Xx} (dy|x) $$
	for	$x\in \Xx$   and $C \in \Sigma_\Zz$. 
	It is well-known that the composition is associative.

	- We  denote  by $\Meas(\Xx, \Yy)$  the set  of  all measurable mappings from   a measurable  space  $\Xx$ to a measurable space $\Yy$, and by $\Probm (\Xx, \Yy)$  the set  of all  probabilistic  morphisms    from $\Xx$ to $\Yy$.  We regard  $\Meas (\Xx, \Yy)$ as a subset  of $\Probm (\Xx, \Yy)$, identifying  $ \Yy \ni  y $ with the Dirac measure $\delta_y \in \Pp (\Yy)$. This is possible, since the  Dirac map $\Yy \to \Pp (\Yy), y \mapsto \delta_y$, is measurable \cite{Lawvere1962},\cite[Theorem 1]{Giry1982}.

	- We denote  by $\Yy ^\Xx$  the set  of all mappings  from $\Xx$ to $\Yy$.  If $\Yy$  is a  measurable space, then  $\Yy^\Xx$    is a measurable space   with the    cylindrical $\sigma$-algebra  denoted by $\Sigma_{cyl} (\Yy^\Xx)$. 
	
	- For  any $\Xx$ we denote by $\Id_\Xx$ the  identity map on $\Xx$. For a   product   space  $\Xx \times \Yy$ we denote  by $\Pi_\Xx$ the canonical projection   to the    factor  $\Xx$.
	
	- For any $T \in \Probm (\Xx, \Yy)$ the linear  mapping 
	$$ P_* T: \Pp (\Xx)\to \Pp (\Yy),$$ 
	\begin{equation}\label{eq:Mhomomorphism}
		P_* T (\mu)  (B): = \int _\Xx  \overline T(B|x) \, d\mu (x),   \, \mu \in \Pp (\Xx), \, B \in \Sigma_\Yy,
	\end{equation}
	is  measurable \cite{Lawvere1962}, \cite[Theorem 1]{Giry1982}.  If $\Xx$ and $\Yy$ are  Polish spaces, then   $\Pp (\Xx)$   and $\Pp (\Yy)$  are Polish spaces endowed with the  weak*-topology $\tau_w$,  and  their $\sigma$-algebra  $\Sigma_w$ is the Borel $\sigma$-algebra $\Bb (\tau_w)$. Furthermore, $P_* T$  is a  $(\tau_w, \tau_w)$-continuous map \cite[Theorem 1]{Giry1982}.  Moreover, for any $T_1 \in \Probm (\Xx_1, \Xx_2)$, $T_2 \in \Probm (\Xx_2, \Xx_3)$  we have \cite{Lawvere1962}, \cite[Lemma 5.10]{Chentsov72},\cite[Theorem 1]{Giry1982}, \cite[Proposition 5]{JLT2021}
	\begin{equation}\label{eq:giryf}
		P_* (T_2 \circ T_2) = P_* T_2 \circ P_* T_1.
	\end{equation}
	   We  also use  the abbreviation $T_*$   for $P_* T$.
	   
	   - For  $T_i \in
	   \Probm (\Xx_i, \Xx_{i+1})$, $ i = 1, 2$, we have \cite[Lemma 5.5]{Chentsov72}
	   \begin{equation}\label{eq:tcompose}
	   	\overline{T_2 \circ T_1} =  (T_2) _* \circ \overline {T_1}.
	   \end{equation}

	-  For any $k \in \N^+$ the multiplication mapping
	\begin{equation}\label{eq:frakmk}
		\mathfrak m^k:\prod_{i =1}^k\big (\Pp(\Xx_1) , \Sigma_w\big ) \to   \Big (\Pp \big (\prod_{i=1}^k \Xx_i\big ), \Sigma_w \Big ), \,  (\mu_1, \ldots, \mu_k) \mapsto \otimes_{i=1}^k \mu_i
	\end{equation}
	is  measurable \cite[Proposition 2.1 (1)]{Le2025}.

	-  For  a        probabilistic  morphism  $T: \Xx \leadsto \Yy$   the    graph $\Gamma  _T: \Xx \leadsto  \Xx \times  \Yy$ of $T$  is defined as follows:
	$$\overline{\Gamma _T}  (x) : = \mathfrak m ^2 (\overline{\Id_\Xx}, \overline T).$$
	Note that  $\overline{\Id_\Xx} = \delta \circ \Id_\Xx$, where  
	$$\delta:\Xx \to \Pp (\Xx)$$
	 is the  measurable map assigning   $x \in \Xx$ to the Dirac measure $\delta _x$ concentrated at $x$.  For any $\mu \in \Pp (\Xx)$, $A\in \Sigma_\Xx$, $B \in \Sigma_\Yy$  we have \cite[(2.13)]{Le2025}
	 \begin{equation}\label{eq:pgraph}
	 	 (\Gamma_T)_* \mu  (A \times B) = \int _\Xx \overline {\Gamma_T} (x) (A\times  B)d\mu (x) = \int_A \overline  T  (B|x)\, d\mu (x).
	 	\end{equation}
	 By \cite[Lemma 2.10 (1), (2.15)]{Le2025}, we have the following decomposition:
	 \begin{equation}\label{eq:tdecomp}
	 	T = \Pi_\Yy \circ \Gamma_T.
	 	\end{equation}
	
	In \cite[Lemma 2.10(2)]{Le2025}  L\^e  proved  the following   formula for the graph of  a composition  of  probabilistic morphisms  $p_1 : \Xx \leadsto \Yy$  and $p_2 : \Yy \leadsto \Zz$
	\begin{equation}\label{eq:graphcomp}
		\Gamma_{p_2 \circ p_1} = (\Id_\Xx \times p_2 ) \circ \Gamma_{p_1}.
	\end{equation}
	-    A {\it Bayesian statistical model}  is  a   quadruple $(\Theta, \mu_\Theta, \pb, \Xx)$,  where  $(\Theta,\mu_\Theta)$  is a probability space,   and  $\pb\in \Meas\big( \Theta , \Pp(\Xx)\big)$.
	The {\it predictive  distribution} $\mu_\Xx \in \Pp (\Xx)$  of    a    Bayesian statistical model $(\Theta, \mu_\Theta, \pb, \Xx)$ is defined as the {\it prior marginal probability} of  $x$, i.e., $\mu_\Xx : = (\Pi_\Xx)_*\mu$,  where $\mu: = (\Gamma_{\underline \pb})_*\mu_\Theta \in \Pp (\Theta\times \Xx) $ is the joint distribution of $\theta\in \Theta$ and $x\in \Xx$  whose regular conditional probability measure  with respect to the projection $\Pi_\Theta: \Theta \times \Xx \to \Theta$ is $\pb:\Theta \to \Pp(\Xx)$.
	A {\it Bayesian inversion}  $ \qb: =\qb (\cdot \|\pb, \mu_\Theta)\in \Meas\big( \Xx, \Pp (\Theta)\big)$  of a Markov kernel  $\pb\in \Meas\big ( \Theta, \Pp (\Xx)\big )$   relative to  $\mu_\Theta$ is a Markov kernel  such that
	\begin{equation}\label{eq:bayesinv0}
		(\sigma_{\Xx, \Theta})_*	(\Gamma_{\underline \qb})_*\mu_\Xx = (\Gamma_{\underline \pb})_* \mu_\Theta,
	\end{equation}
	where  $\sigma_{\Xx, \Theta}: \Xx\times \Theta \to \Theta\times \Xx$ is defined by $  (x, \theta) \mapsto  (\theta, x)$.
	
	We also write $\qb (\cdot \|\pb, \mu_\Theta)$ as $\qb (\cdot \|\mu_\Theta)$ if $\pb$ is fixed and   no confusion  can occur.

	For $X_m : = (x_1, \ldots, x_m)\in \Xx^m$, we denote by 
	$$E_{X_m}:\Pp (\Yy) ^\Xx \to \Pp (\Yy) ^m, h \mapsto \big (h(x_1),\ldots, h(x_m))\in \Pp (\Yy)^m$$
	the evaluation mapping.
	
	For  $S_n = \big ((x_1, y_1), \ldots , (x_n, y_n)\big) \in  (\Xx \times \Yy)^n$,  we denote  by $\Pi_\Xx (S_n)$  the $\Xx^n$-component  of
	$S_n$, namely  $\Pi_\Xx (S_n) =  (x_1, \ldots,  x_n)\in \Xx^n$. Similarly,  $\Pi_\Yy (S_n) = (y_1, \ldots, y_n)\in \Yy^n$.  Now we    recall the  solution  of the problem SBI
	in \cite[Definition 3.2]{Le2025}.

	\begin{definition}\label{def:fposteriord}  A {\it   Bayesian learning  model for the supervised  inference problem SBI} consists of  a      quadruple $(\Theta, \mu_\Theta, \pb , \Pp (\Yy)^\Xx)$, where  $(\Theta, \mu_\Theta)$ is a probability parameter space,  and $\pb: \Theta \to \Pp (\Yy)^\Xx$  is a   measurable mapping. 
		
	(1)	 For $X_m = (x_1, \ldots, x_m)\in  \Xx^m$, the sampling operator $\pb_{X_m}: = \mathfrak m^m \circ E_{X_m}\circ \pb: \Theta \to \Pp (\Yy^m)$ 
		parameterizes   the sampling distributions  of  $Y_m = (y_1, \ldots, y_m)\in \Yy^m$, where $y_i$ is  a label of $x_i$, with certainty encoded in  $\mu_\Theta$.

		(2) For  a training sample  $S_n\in (\Xx \times \Yy)^n$,  the  {\it posterior  distribution} $ \mu_{\Theta|S_n} \in \Pp (\Theta)$ after  seeing $S_n$ is  the value $ \qb_{\Pi_\Xx (S_n)}\big(\Pi_\Yy  (S_n)\big) $  where  $\qb_{\Pi_\Xx (S_n)}:\Yy^n \to \Pp (\Theta) $  is  a  Bayesian  inversion of  the     Markov kernel $\pb_{\Pi_\Xx (S_n)}:\Theta \to \Pp (\Yy^n)$  relative  to  $\mu_\Theta$.
		
		(3) For $T_m = (t_1, \ldots , t_m) \in \Xx^m$, the {\it posterior  predictive distribution}  $\Pp_{T_m|S_n, \mu_\Theta}\in \Pp (\Yy^m) $ of the $m$-tuple  $(y'_1, \ldots, y'_m)$  where $y'_i$ is the label of $t_i$, given a  training data   set $S_n \in (\Xx \times \Yy)^n$,
		is defined  as the  predictive  distribution  of  the    Bayesian statistical  model $(\Theta,   \mu_{\Theta|S_n},  \pb_{T_m}, \Yy^m)$, i.e., 
		\begin{equation}\label{eq:posteriorsuper}
			 \Pp  _{T_m |  S_n, \mu_\Theta}: = (\underline{\pb_{T_m}})_* \mu_{\Theta|S_n} \in \Pp (\Yy^m).
		\end{equation}
		(4) The aim  of  a learner  is to  estimate  and approximate   the value  of the posterior   predictive    distribution  $\Pp  _{T_m |  S_n, \mu_\Theta}$.
	\end{definition}
	
	\begin{remark}\label{rem:uniqueness} (1)  A Bayesian  inversion  $\qb_{X_n}: \Yy ^n \to \Pp (\Theta)$  of  the     Markov kernel $ \pb_{X_n}:\Theta \to \Pp (\Yy^n)$  relative  to  $\mu_\Theta$ is  defined  	$(\underline{\pb_{X_n}})_*\mu_\Theta$-a.s. uniquely in  weak sense. \footnote{See \cite[Theorem 2.9]{Le2025}, for  the precise  definition. In this paper,  we need only Formula  \eqref{eq:bayesinv0}  for a  characterization  of   a Bayesian  inversion.}   In  Theorem \ref{thm:posspred}  below, we shall show that, given inputs $X_n = (x_1, \ldots, x_n)  \in \Xx^n$,   for each $T_m \in  \Xx^m$, the  posterior  predictive  distribution  
		\begin{equation}\label{eq:posterior}
			\Pp^n_m:   \Yy ^n \to \Pp (\Yy^m), \,  (y_1, \ldots,  y_n) \mapsto \Pp_{T_m|  \big((x_1, y_1),\ldots (x_n, y_n)\big), \mu_\Theta}, 
	\end{equation}
	is  a regular  conditional probability measure  of the joint  distribution\\ $(\underline{\mathfrak m^2 (\pb_{T_m},\pb_{X_n})}_* \mu_\Theta \in \Pp (\Yy ^m \times \Yy^n)$.  Hence,  $\Pp^n_m$ is  defined  	$(\underline{\pb_{X_n}})_*\mu_\Theta$-a.s. uniquely	 in the weak sense,  independently from the choice  of a  Bayesian  inversion $\qb_{X_n}: \Yy ^n \to \Pp (\Theta)$.
	
	(2)  If $\# (\Xx) =1$,   our Bayesian  learning model  is a classical Bayesian  statistical model $\big(\Theta, \mu_\Theta, \pb, \Pp (\Yy)\big)$  for    Bayesian inference  under the assumption of  conditionally i.i.d.  data $y \in \Yy$.
	
	(3) We showed in \cite{Le2025} that classical Bayesian regression  learning is a particular  case  of    Bayesian  supervised  learning  in the sense of Definition \ref{def:fposteriord}, see  Section \ref{sec:pred} and  Theorem \ref{thm:posspredr}.
	
	(4) It is easy to verify, e.g., by using Equation \eqref{eq:marginal1} in the next Section, that  our solution of  Problem SBI satisfies the following  consistency  property.  For any $T_m  =   (t_1, \ldots,  t_m)\in \Xx^m$, $S_n \in  (\Xx \times \Yy)^n$, and    $t_{m+1}\in \Xx$  we have
	\begin{equation}\label{eq:consb}
		\Pp_{T_m|S_n,\mu_\Theta} = (\Pi_{\Yy^m})_* \Pp_{ T_m, t_{m+1}|S_n,\mu_\Theta}.
	\end{equation}
	In particular,  to solve   Problem  SBI  for all $m\in \N^+$, it suffices to solve  the problem for the case $m = \infty$.
	\end{remark}
	
	In \cite[Proposition 3.4]{Le2025}  we  showed that the quadruple   $\big( \Pp (\Yy)^\Xx, \mu,\\ \Id _{\Pp (\Yy) ^\Xx}, \Pp(\Yy)^\Xx\big)$,  where $\mu \in \Pp (\Pp (\Yy)^\Xx)$, is a universal  Bayesian learning model in the sense  of Definition \ref{def:fposteriord}     for solving   the  problem SBI.

	\section{Bayesian inversions  in Bayesian  supervised  learning models}\label{sec:bayin}
	
	In this Section,  first   we  prove  Theorems  \ref{thm:binv} and \ref{thm:projbi} for computing   Bayesian  inversions in supervised learning model.
	We illustrate these theorems with   Example \ref{ex:dirichleti}.

	\begin{theorem}[Online  formula for Bayesian inversion]\label{thm:binv}  Let  $(\Theta, \mu_\Theta, \pb, \Pp (\Yy)^\Xx)$  be a  Bayesian  model  for supervised  learning.   Let $S_n =  \big  (( x_1, y_1), \ldots, (x_n, y_n) \big)\in (\Xx \times \Yy)^n$  and $S_{n-1} = \big  (( x_1, y_1), \ldots, (x_{n-1}, y_{n-1}) \big)$. 
		 Then  a Bayesian inversion $\qb_{\Pi_\Xx (S_n)}(\cdot\| \mu_\Theta): \Yy^n \to \Pp (\Theta)$ of the Markov kernel $\pb_{\Pi_\Xx (S_n)}: \Theta \to \Pp (\Yy^n)$ relative to $\mu_\Theta$ can be found     by the following formula:
		\begin{equation}\label{eq:recursive}
			\qb_{\Pi_\Xx (S_n)} (y_1, \ldots, y_n \| \mu_\Theta):= \qb_{x_n} \bigl(y_n\, \|\, \qb_{ \Pi_\Xx (S_{n-1})} (y_{1}, \ldots, y_{n-1}\|\mu_\Theta )\bigr).	
		\end{equation}
	\end{theorem}

We abbreviate  $	\qb_{\Pi_\Xx (S_n)} (\cdot\| \mu_\Theta)$  as $\qb_{\Pi_\Xx (S_n)}$. For the  proof of   Theorem  \ref{thm:binv} we need  Lemma  \ref{lem:decompev} below  stating that  $\underline{\pb_{x_n} }\circ \underline{\qb_{\Pi_\Xx (S_{n-1})}}: \Yy^{n-1} \leadsto  \Yy$ is a   regular conditional probability measure   for 	$\big(\underline{\pb_{\Pi_\Xx (S_{n})}}\big)_*\mu_\Theta \in \Pp (\Yy ^{n})$ with respect to the projection $\Pi_{\Yy^{n-1}}: \Yy^n \to \Yy^{n-1}$. 
	\[
\xymatrix@1{
	(\Theta,\mu_\Theta) \ar@{~>}[dd]_{\underline{\pb_{x_n}}}\ar@/_1pc/@{~>}[ddrrrr]_{\underline{\pb_{\Pi_\Xx (S_{n-1})}}} \ar@{~>}[rrrr]^
	{\underline{\pb_{\Pi_\Xx (S_n)}}} & & & & \Yy^{n}\ar[dd]^{\Pi_{\Yy^{n-1}}} \\
	& &&&\\
	\Yy  &&&& \ar@/_1pc/@{~>}[uullll]_{\underline{\qb_{\Pi_\Xx (S_{n-1})}}} \Yy^{n-1}.
}
`\]

\begin{lemma} \label{lem:decompev}  Assume the  condition  of Theorem \ref{thm:binv}.  Then  we  have
	\begin{equation}\label{eq:marginal}
		 \big(\underline{ \pb_{\Pi_\Xx (S_{n-1})}}\big)_* \mu _\Theta  =  (\Pi_{\Yy^{n-1}})_* \big(\underline{\pb_{\Pi_\Xx (S_n)}}\big)_*\mu_\Theta, 
		 \end{equation}
\begin{equation}\label{eq:n+1}
		\big(\underline{\pb_{\Pi_\Xx (S_n)}}\big)_*\mu_\Theta = \big(\Gamma_{\underline {\pb_{x_n}} \circ \underline{\qb_{\Pi_\Xx (S_{n-1})}}}\big)_*  \big(\underline{ \pb_{\Pi_\Xx (S_{n-1})}}\big)_* \mu _\Theta.
\end{equation}
\end{lemma}
\begin{proof}[Proof of Lemma \ref{lem:decompev}]  Lemma \ref{lem:decompev} is a particular case of Proposition \ref{prop:decompev} below.
\end{proof}

	\begin{proposition} \label{prop:decompev}  
		Assume that  $(\Theta, \mu_\Theta)$ is a probability space, $\Xx, \Yy$ are measurable spaces. Let  $\pb _\Xx \in \Meas \big(\Theta, \Pp (\Xx)\big)$,  $\pb _\Yy \in \Meas \big(\Theta, \Pp  (\Yy)\big)$.  If $\qb_\Yy \in \Meas (\Yy, \Pp  (\Theta))$ is a Bayesian inversion of  $\pb _\Yy$ relative to $\mu_\Theta$,  then   we  have
		\begin{equation}\label{eq:marginal1}
		\big(\underline{ \pb_{\Yy}}\big)_* \mu _\Theta  =  (\Pi_{\Yy})_* \big(\underline{\mathfrak m ^2 (\pb_\Yy , \pb_\Xx)}\big)_*\mu_\Theta , 
	\end{equation}	
		\begin{equation}\label{eq:n+1g}
			\big(\underline{\mathfrak m ^2 (\pb_\Yy , \pb_\Xx)}\big)_*\mu_\Theta = \big(\Gamma_{\underline {\pb_\Xx} \circ \underline{\qb_\Yy}}\big)_*  \big(\underline{ \pb_\Yy}\big)_* \mu _\Theta.
		\end{equation}
			\[
		\xymatrix@1{
			(\Theta,\mu_\Theta) \ar@{~>}[dd]_{\underline{\pb_{\Xx}}}\ar@/_1pc/@{~>}[ddrrrr]_{\underline{\pb_\Yy}} \ar@{~>}[rrrr]^
			{\underline{\mathfrak m^2 (\pb_\Yy, \pb_\Xx)}} & & & & \Yy \times \Xx\ar[dd]^{\Pi_{\Yy}} \\
			& &&&\\
			\Xx  &&&& \ar@/_1pc/@{~>}[uullll]_{\underline{\qb_\Yy}} \Yy.
		}
		\]
	\end{proposition}

	\begin{proof}[Proof of Proposition \ref{prop:decompev}] (1)  To prove \eqref{eq:marginal1}, it  suffices  to show  that
		\begin{equation}\label{eq:marginal2}
			\pb _\Yy   = \overline{\Pi_\Yy \circ \underline{\mathfrak m ^2 (\pb_\Yy , \pb_\Xx)}}.
		\end{equation}
		Using  Composition Formula \eqref{eq:tcompose}, we conclude that Equation  \eqref{eq:marginal2} is  equivalent  to the following:
		\begin{equation}\label{eq:marginal3}
		\pb _\Yy (\theta)   = (\Pi_{\Yy})_* \mathfrak m ^2 (\pb_\Yy , \pb_\Xx) (\theta) \;  \forall \theta \in \Theta. 
	\end{equation}	
	Equation \eqref{eq:marginal3} holds since $\mathfrak m ^2 (\pb_\Yy , \pb_\Xx) (\theta) =  \pb_\Yy (\theta)\otimes \pb_\Xx(\theta)$.
	
	(2)	  Using  Formula \eqref{eq:graphcomp}
	for   the graph of a composition of probabilistic  morphisms,  we obtain
	\begin{align}
		\big(\Gamma_{\underline {\pb_{\Xx}} \circ \underline{\qb_\Yy}}\big)_*  \big(\underline{ \pb_\Yy}\big)_* \mu _\Theta\nonumber\\
		=  \big(\Id_{\Yy}\times \underline{\pb_{\Xx}}\big)_* \big(\Gamma_{\underline{\qb_\Yy}})_*
		\big(\underline{\pb_\Yy}\big)_* \mu _\Theta.
		\label{eq:decompev1g}
	\end{align}	
	Taking into account  that $\qb_\Yy: \Yy \to \Pp (\Theta)$ is a Bayesian inversion of  $ \pb_\Yy: \Theta \to \Pp (\Yy)$ relative to $\mu_\Theta$, 
	we obtain Equation \eqref{eq:n+1g} from  \eqref{eq:decompev1g}:
	\begin{align*}
	\big(\Gamma_{\underline {\pb_\Xx} \circ \underline{\qb_\Yy}}\big)_*  \big(\underline{ \pb_\Yy}\big)_* \mu _\Theta 
 =  \big(\Id_{\Yy}\times \underline{\pb_\Xx	}\big)_* (\sigma_{\Theta, \Yy})_*  \big (\Gamma_{\underline{\pb_\Yy}}\big)_*\mu_\Theta\nonumber\\
 = \big(\underline{\mathfrak m ^2 (\pb_\Yy , \pb_\Xx)}\big)_*\mu_\Theta.
\end{align*}
	\end{proof}

	\begin{proof}[Proof of Theorem  \ref{thm:binv}] 
		By   Decomposition Formula \eqref{eq:tdecomp},  we have
		$$(\underline{\pb_{\Pi_\Xx (S_n)}})_*\mu_\Theta = (\Pi _\Yy)_* (\Gamma _{\underline{\pb_{\Pi_\Xx (S_n)}}})_* \mu_\Theta.$$
		Hence, to prove that $\qb_{\Pi_\Xx (S_n)}: \Yy^{n } \to \Pp (\Theta)$ defined  by \eqref{eq:recursive} is a Bayesian inversion of $\pb_{\Pi_\Xx (S_n)}: \Theta \to \Pp (\Yy^n)$, it suffices to show that
		$$ (\sigma_{\Theta, \Yy^n})_*  (\Gamma _{ \underline{\pb_{\Pi_\Xx(S_n)}}})_* \mu_\Theta  =  (\Gamma_{ \underline{\qb_{\Pi_\Xx (S_n)}}})_* (\pb_{\Pi_\Xx(S_n)} )_* \mu_\Theta.$$
		By \eqref{eq:pgraph}, it suffices  to show that  for  any   $A_n \in \Sigma_\Yy$, $A_{n-1} \in \Sigma_{\Yy^{n-1}}$,  $B \in \Sigma_\Theta$ we have
		\begin{align}
			\int_{B}  \pb_{x_n} (A_n|\theta)\pb_{\Pi_\Xx (S_{n-1})} (A_{n-1}|\theta)d\mu_\Theta (\theta)\nonumber\\
			= \int_{A_{n-1} \times  A_n} \qb_{x_n} \big(B| y_n\|\qb_{\Pi_\Xx (S_{n-1})} (y_1, \ldots, y_{n-1}) \big)d(\underline{\pb_{\Pi_\Xx (S_n)}} )_*\mu_\Theta  (y_1,\ldots,  y_n).\label{eq:binvern}
		\end{align}
		
		By Lemma \ref{lem:decompev},  the  Markov kernel  $\overline{ \underline{\pb_{x_n}} \circ  \underline{\qb_{\Pi_\Xx(S_{n-1})}}}: \Yy^{n-1}  \to\Pp ( \Yy)$ is a  regular  conditional  probability measure  for  the joint  distribution  $(\underline{\pb_{\Pi_\Xx(S_n)}})_* \mu_\Theta \in \Pp (\Yy^n)$ with respect to the projection $\Pi_{\Yy^{n-1}} : \Yy^n \to \Yy^{n-1}$.  Taking into account  Composition Formula  \eqref{eq:tcompose},  we have 
		
	\begin{equation}\label{eq:decomp1} \overline{ \underline{\pb_{x_n}} \circ  \underline{\qb_{\Pi_\Xx(S_{n-1})}}} = (\underline {\pb_{x_n}})_* \circ  \qb_{\Pi_\Xx (S_{n-1})}.
\end{equation}
Applying the  disintegration formula,  and taking into account \eqref{eq:decomp1}, we  verify    \eqref{eq:binvern}  as follows
		
		\begin{align}
			\int_{B}  \pb_{x_n} (A_n|\theta)\pb_{\Pi_\Xx (S_{n-1})} (A_{n-1}|\theta)d\mu_\Theta (\theta)\nonumber\\
			\stackrel{?}{=} \int_{A_{n-1}} \int_{  A_n} \qb_{x_n} \big(B| y_n\|\qb _{\Pi_\Xx (S_{n-1})}(y_1, \ldots , y_{n-1})\big)  d(\underline {\pb_{x_n}})_* \nonumber\\
			\big(\qb_{\Pi_\Xx (S_{n-1})} (y_1, \ldots, y_{n-1})\big) (y_n) d(\underline{\pb_{\Pi_\Xx (S_{n-1})}} )_*\mu_\Theta  (y_1, \ldots, y_{n-1})\nonumber\\
			= \int_{A_{n-1}} \int _B\pb_{x_n} (A_n|\theta) d\qb_{\Pi_\Xx (S_{n-1})} (\theta|y_1, \ldots, y_{n-1})d(\underline{\pb_{\Pi_\Xx(S_{n-1})}})_* \mu_\Theta  (y_1, \ldots, y_{n-1}).\label{eq:binvern1}
		\end{align}
		The last equality holds  since  
		$$\qb_{x_n}  \big(\cdot |\qb_{\Pi_\Xx (S_{n-1})} (y_1, \ldots, y_{n-1})\big): \Yy \to \Pp (\Theta)$$
		 is  a Bayesian inversion of $\pb_{x_n}:\Theta \to \Pp (\Yy)$ relative to $\qb _{\Pi_\Xx (S_{n-1})} (y_1, \ldots, y_n)\in \Pp  (\Theta)$.
		
		Since  $\pb_{x_n} (A_n|\cdot) \in \Ff_b (\Theta)$,  fixing $A_{n-1} $ and $B$,  we  extend  the LHS and RHS of \eqref{eq:binvern1} as  linear functions  on  $\Ff_b (\Theta)$.  Since $\Ff_s (\Theta)$ is  dense   in $\Ff_b (\Theta)$ in the sup-norm, to prove
		\eqref{eq:binvern},  it suffices  to show that for any $B' \in \Sigma_\Theta$ we have
		
		\begin{align}
			\int_{ B\cap B'} \pb_{\Pi_\Xx (S_{n-1})} (A_{n-1}|\theta)d\mu_\Theta (\theta)\nonumber\\
		 = \int_{A_{n-1}} \int_{ B \cap B'} d\qb_{\Pi_\Xx (S_{n-1})} (\theta|y_1, \ldots, y_{n-1})d(\underline{\pb_{\Pi_\Xx (S_{n-1})}})_*\mu_\Theta (y_1, \ldots, y_{n-1}).\label{eq:binvern2}
		\end{align}
		
		Equation \eqref{eq:binvern2}  holds  since  $\qb_{\Pi_{\Xx (S_{n-1})}}: \Yy^{n-1} \to \Pp (\Theta)$ is a Bayesian  inversion  of  $\pb_{\Pi_\Xx (S_{n-1})}: \Theta \to \Yy^{n-1}$ relative to $\mu_\Theta$.

	\end{proof}

	For  a finite  set $\Xx$  and a measurable space $\Yy$, we propose   another method for   computing     Bayesian  inversions on universal  Bayesian   models  $\big(\Pp (\Yy) ^\Xx, \mu,\\ \Id_{\Pp  (\Yy) ^\Xx}, \Pp(\Yy) ^\Xx\big)$ using projective  limits.
	
	We denote by $\pi (\Yy)$ the set of all   finite partitions of $\Yy$  into measurable subsets.
	
	   Denote  by $|(A)|$ the  size of  a  finite  partition  $(A) =  \dot \cup_{ i=1} ^k A_i\in \pi(\Yy)$, and by $\Om_{ (A)}:= \{  A_1, \ldots, A_k\}$  the finite   set of  size  $|(A)|$ associated  with $(A)$.
	  Each  partition  $  (A) =  \dot \cup_{ i=1} ^k A_i\in \pi(\Yy)$  is defined uniquely by   a surjective measurable map 
	$$\pi_{(A)} : \Yy \to \Om_{(A)}$$ 
	that maps  $A_i \ni y $ to  $A_i$. 
	If   a partition $ (A) = (A_1, \ldots, A_k)$ of size $k$  of $\Yy$  is a refinement   of  a partition $ (A') =(A'_1, \ldots, A_l')$ of  size $l\le k$ we write  $(A')\le (A)$.  Then there  exists
	a map  $$\pi ^{(A)}_{ (A')}: \Om _{(A)} \to \Om_{(A')}$$ such that
	\begin{equation}\label{eq:compose1}
		\pi_{(A')} = \pi^{(A)}_{ (A')} \circ \pi_{(A)}.
	\end{equation}
	Thus  $(\pi  (\Yy), \le)$ is a  directed  set of finite  (measurable) partitions  of   $\Yy$. 
	 
	\begin{theorem}\label{thm:projbi} Let   $\Yy$ be a measurable  space,  and  $\Xx : = \{ x_1, \ldots, x_n\}$  a finite  set  and  $  X_k \in \Xx ^k$.  Let  $\mu \in \Pp (\Pp  (\Yy)^\Xx)$.  Assume that  for any $ (A)\in \pi(\Yy)$  there exists a   Bayesian  inversion  $\qb ^{(A)}_{X_k}: \Om _{(A) } ^k \to  \Pp  \big (\Pp  (\Om _{ ( A)}) ^\Xx)$  of $\mathfrak m^k \circ E_{X_k}: \Pp(\Om_{(A)}) ^\Xx \to  \Pp(\Om_{  (A)} ^k)$  relative to   $P_*  (\pi_{(A)})_* ^\Xx (\mu)\in \Pp (\Pp (\Om_{ (A)})^\Xx)$
		such that
		   for any  $ (A) \le (B)  \in  \pi (\Yy)$  the  following diagram is commutative:
		
		$$ 
		\xymatrix{\Pp \big(\Pp (\Om_{ (B)})^\Xx\big)  \ar[d]_{P_*(\pi^{(B)}_{ (A)}) ^\Xx_*} & & \ar[ll]_{\qb^{(B)}_{X_k}} \Om_{ (B)} ^k  \ar[d]^{ ( \pi^{(B)}_{ (A)})^k}\\
		\Pp \big (	\Pp (\Om_{ (A)})^\Xx\big) &  &\ar[ll]_{\qb ^{(A)}_{X_k}}\Om_{ (A)} ^k 	.
		}
		$$
		Assume  that there  exists  a  map $\qb_{X_k}: \Yy^k \to \Pp (\Pp (\Yy)^\Xx)$ such that   for any  $ (A) \in \pi (\Yy)$ the following  diagram is commutative  for any $ (A)\in \pi (\Yy)$.
		\begin{equation}\label{eq:comm2}
		\xymatrix{\Pp \big (\Pp (\Yy)^\Xx \big) \ar[d]_{P_* (\pi_{ (A)}) ^\Xx_*} & & \ar[ll]_{\qb_{X_k}} \Yy ^k  \ar[d]^{(\pi_{ (A)})^k}\\
		\Pp \big(	\Pp (\Om_{ (A)})^\Xx\big) & & \ar[ll]_{\qb ^{(A)}_{X_k}}\Om_{ (A)} ^k 	.
		}
		\end{equation}
		 Then   $\qb_{X_k} $ is a  Bayesian inversion  of  $\mathfrak m ^k \circ E_{X_k}:\Pp (\Yy)^\Xx  \to \Pp(\Yy^k)$ relative to  $\mu$.
	\end{theorem}
	\begin{proof}   To prove  Theorem \ref{thm:projbi}, it  suffices to show that  $\qb_{X_k}$ is  a measurable map and  for any $ \Aa= \Aa_1 \times \ldots \times  \Aa_n  \in \Sigma_{\Pp (\Yy) ^\Xx}$, $\Aa_i \in \Sigma_{\Pp  (\Yy)}$,  $ B = B_1 \times \ldots \times  B_k \in  \Sigma_{\Yy^k}$, $B_j \in \Sigma _\Yy$,  we have 
		\begin{equation}\label{eq:lim1}
			 (\Gamma_{\underline{\qb_{X_k}}})_*
			 (\underline{\mathfrak m^k \circ E_{X_k}})_*\mu  (B \times \Aa)= (\Gamma_{\underline{\mathfrak m^k\circ E_{X_k}}})_* \mu  (\Aa \times B) .
		\end{equation}
	
Recall  that  $\Sigma_{\Pp (\Yy)}$  is generated  by    subsets  $ e ^{-1}_{A} (C)$   where  $A\in \Sigma_\Yy$,  $ C \in \Bb (\R)$, and $  e_{A}: \Pp  (\Yy) \to \R$ is defined  by  $\mu \mapsto  \mu (A)$. 
	\begin{lemma}\label{lem:inverse}  Let $\Aa_i = e^{-1}_{A_i} (C_i)\in \Sigma_{\Pp (\Yy)}$ for $i \in \overline{1, n}$.
		Then there  exists   a  finite  partition  $ (\tilde A)\in \pi  (\Yy)$  and  a    subset   $ \Ss_1, \ldots, \Ss_n  \in \Sigma_{ \Pp \big( \Om _{   (\tilde A)}\big)}$ such that 
		\begin{equation}\label{eq:inverse}
		 \big ((\pi_{ (\tilde A)}) _* \big) ^{-1} (\Ss_i) = \Aa_i  \text{ for  all } i \in \overline{1, n}.
			\end{equation}
		Hence,    
		 $\Sigma_{\Pp(\Yy) ^\Xx}$  is generated  by  subsets 
		$\Pi_{ j=1} ^n\big(\pi_{(A)})_*^{-1} (ev_{S_i^j}^{-1} (C_i ^j))$ where  $C^j_i \in \Bb (\R)$ and  $S_i^j \subset  \Om _{ (A)}$, $ (A)\in \pi (\Yy)$.
	\end{lemma}
	
	\begin{proof}[Proof of Lemma \ref{lem:inverse}]  Let  $ (\tilde  A)\in \pi (\Yy)$ be a finite partition such that   for any  $i\in \overline{1, n}$  there  exists  a subset  $S_i \in \Om_{ |(\tilde  A)|}$ such  that
		\begin{equation}\label{eq:inverse1}
			\pi_{ (\tilde A)}  ^{-1} (S_i) = A_i.
		\end{equation}
		Then we   have the following commutative diagram  for any  
		$ i \in \overline{1, n}$:
		\begin{equation}\label{eq:comm3}
	\xymatrix{
		\Pp (\Yy)\ar [rr]^{e_{ A_i}} \ar[d]_{\big(\pi_{ (\tilde A)}\big)_*} &&  \R\\
		 \Pp ( \Om_{ (\tilde A)} )\ar[urr] ^{e_{S_i}}.  & &
	}	
		\end{equation}
		It follows that  $\Aa_i =  e_{A_i} ^{-1} (C_i) =  \Ss_i : = \Big(\big(\pi_{ (\tilde A)}\big)_*\Big) ^{-1} ( e^{-1}_{S_i} (C_i))$. This   proves  \eqref{eq:inverse}. The  last  assertion of Lemma \ref{lem:inverse} follows  immediately.
		\end{proof}
	
	{\it Completion  of the proof of  Theorem \ref{thm:projbi}.}     To prove that $\qb _{X_k}$ is measurable, it suffices to show that for any  $\Aa  = \Aa_1 \times \ldots \times \Aa_n  \in \Sigma _{\Pp  (\Yy) ^\Xx}$, $\Aa_i \in \Sigma _{\Pp (\Yy)}$, the composition 
	$ev_A \circ \qb_{X_k}: \Yy^k \to \R$ is  measurable.    By Lemma \ref{lem:inverse}, taking into account   the commutative diagrams \ref{eq:comm2}  and \ref{eq:comm3}  we have 
	$$ev_A \circ \qb_{X_k} = ev_{\tilde A} \circ  \qb_{X_k}   ^{(\tilde A)}$$
	which is   measurable  by   the assumption  of Theorem \ref{thm:projbi}.
	 To complete the proof of Theorem \ref{thm:projbi},  it suffices   to  prove \eqref{eq:lim1}.     By Lemma \ref{lem:inverse},    the RHS  of \eqref{eq:lim1}    equals  
	$(\Gamma_{\underline {\mathfrak m^k \circ E_{X_k}}}  P_*  (\pi_{(\tilde A)}) ^\Xx_* \mu (\Ss_1\times  \ldots  \times \Ss_n \times B )$  and the   LHS  of \eqref{eq:lim1}   equals  
	$(\Gamma_{ \underline{\qb ^{ (\tilde A)}_{X_k}}})_* (\underline{\mathfrak  m ^k \circ E_{X_k}})_*  P_*  (\pi_{(\tilde A)}) ^\Xx_* \mu$.   Since  $\qb ^{ (\tilde A)}_{X_k}$ is  a Bayesian  inversion of $\mathfrak  m ^k \circ E_{X_k}$,   \eqref{eq:lim1}  holds.
	\end{proof}
	
	\begin{example}[Posterior distributions of Dirichlet processes]\label{ex:dirichleti}    For  a measurable  space $\Yy$  denote by $ \Mm^* (\Yy)$ the  measurable space of all non-zero finite measures  on $\Yy$ whose  $\sigma$-algebra   is defined  in the  same way as  the $\sigma$-algebra  $\Sigma_w$ on $\Pp (\Yy)$, see \cite[\S 2.1]{JLT2021}. 
		   By \cite[Theorem 4]{JLT2021}, there  exists  a measurable map
		$\Dd: \Mm^* (\Yy) \to \Pp ^2 (\Yy)$ such that     $\Dd (\alpha)$  is the Dirichlet   measure  on $\Pp (\Yy)$ with parameter  $\alpha$  and  for any $\alpha \in \Mm^* (\Yy)$ the following diagram is  commutative
		$$
		\xymatrix{
			\Mm^*(\Yy) \ar[r]^{\Dd}\ar[d]^{M_*(\pi_{(A)})} &  \Pp^2(\Yy)\ar[d]^{P^2_*(\pi_{ (A)})}\\
			\Mm^*(\Om_{(A)}) \ar[r]^{Dir_{(A)}}&  \Pp^2(\Om_{(A)}).
		}
		$$
		Here $Dir_{(A)}  (\beta)$ is the Dirichlet distribution  with parameter $\beta \in \Mm^* (\Om_{(A)})$,  and $M_* (\pi_{ (A)}) : \Mm^* (\Yy) \to \Mm (\Om_{ _{|(A)|}})$ is defined  by the same formula \eqref{eq:Mhomomorphism}.
		Let us consider a Bayesian statistical model  $\big (\Pp (\Yy), \Dd (\alpha),\Id_{\Pp  (\Yy)}, \Yy\big )$ associated  with the case $\# (X) = 1$ and $k =1$ in Theorem \ref{thm:projbi}.  Lemma \ref{lem:inverse} and the above commutative diagram imply that
	the condition  of Theorem \ref{thm:projbi}  holds   for $\big(\Pp  (\Yy), \Dd (\alpha), \Id_{\Pp (\Yy)}, \Pp (\Yy)\big)$.
		It is known  that 
		$$\qb_{ (A) } (\cdot \| Dir  (M_* (\pi_{ (A)})\alpha ) : \Om _{(A)} \to \Pp ^2  (\Om_{ (A)}), x\mapsto   Dir  (M_* (\pi_{ (A)})\alpha + \delta_x)  $$
		is a Bayesian inversion of the Markov kernel  $\Id_{\Pp (\Om _{(A)})}$ relative to $\Dd \big(M_* (\pi_{ (A)} \alpha)\big)$. 
		
		For $(A)\le (B) \in  (\pi (\Yy), \le)$, we  verify immediately   that the following diagram 
		$$ 
		\xymatrix{\Pp^2 (\Om _{(B)}) \ar[d]^{P^2_*\pi  ^{ (B)}_{ (A)}} & & & \ar[lll]_{\qb_{(B)} (\cdot \| \Dd (M_* (\pi_{ (B)} \alpha)))} \Om _{(B)}\ar[d]^{\pi  ^{ (B)}_{ (A)}}\\
			\Pp^2 (\Om _{(A)})& & & \ar[lll]_{\qb _{(A)}(\cdot \| \Dd (M_* (\pi_{ (A)} \alpha)) )}	\Om _{(A)}.
		}
		$$
		is commutative.   Hence, by Theorem \ref{thm:projbi},    the map
		$$ \qb: \Yy \to \Pp  ^2 (\Yy) , y \mapsto  \Dd (\alpha + \delta_y)$$
		is  a Bayesian inversion  of $\Id _{\Pp (\Yy)}$ relative to   $\Dd (\alpha)$.
		Taking into account Theorem \ref{thm:binv},  the map
		$$\qb ^n : \Yy^n  \to \Pp^2 (\Yy),  (y_1, \ldots, y_n) \mapsto \Dd (\alpha + \sum_{i=1} ^n\delta_{y_i})  $$
		is a Bayesian  inversion  of   the   Markov kernel
		$$\mathrm{\Id}^n_{\Pp (\Yy)}:  \Pp (\Yy) \to \Pp  (\Yy^n), \mu \mapsto  \otimes ^n \mu $$
		relative to  $\Dd (\alpha)$.
	\end{example}

	\begin{remark}\label{rem:binvs}
		Let $\Yy$ be a Souslin space and let
		$\mu\in\Pp(\Pp(\Yy)^\Xx)$. For every finite subset
		$A\subset\Xx$ and every $X_m\in A^m$, the space
		$\Pp(\Yy)^A$ is Souslin. Hence, Jost-L\^e-Tran \cite[Theorem 1]{JLT2021}
		provides a Bayesian inversion
		\[
		\qb_{X_m,A}:\Yy^m\longrightarrow\Pp\bigl(\Pp(\Yy)^A\bigr)
		\]
		of $\mathfrak m^m\circ E_{X_m}$ relative to
		$(R_A^\Xx)_*\mu$.
		
		When $\Xx$ is uncountable, these finite-dimensional Bayesian
		inversions need not admit versions forming a pointwise-compatible
		projective system. Therefore we do not assert the existence of a
		Bayesian inversion with values in
		$\Pp(\Pp(\Yy)^\Xx)$ without an additional uniform-version
		hypothesis.
	\end{remark}

	\section{Posterior predictive distributions}\label{sec:pred}
	
	 In this Section, using   Theorem \ref{thm:binv},  we  shall prove   Theorem \ref{thm:posspred} and \ref{thm:posspredr} on     recursive computing  posterior   predictive   distributions in Bayesian  supervised learning.
	
	\begin{theorem}[Posterior  predictive distribution]\label{thm:posspred}  Let   $ (\Theta, \mu_\Theta, \pb, \Pp (\Yy) ^\Xx)$ be a   Bayesian model  for supervised  learning,  $X_n = (x_1, \ldots, x_n) \in \Xx^n$,  and $T_m  =  (t_1, \ldots,  t_m)\in \Xx^m$.   For $Y_n  =(y_1, \ldots, y_n) \in \Yy^n$  let $ S_n(X_n, Y_n) : =  \big( (x_1, y_1), \ldots, (x_n, y_n)\big) \in   (\Xx\times \Yy)^n$.
	
	1)  The  posterior  predictive   distribution $\Pp^n_m: \Yy^n \to \Pp (\Yy^m)$ defined  by   Equation  \eqref{eq:posterior}  is a   regular  conditional probability measure  for the joint distribution    
	$$\mu ^0_{T_m,S_n(X_n, Y_n), \mu_\Theta}: =\underline{\mathfrak  m^2 (\pb_{T_m}, \pb_{X_n})}_* \mu_\Theta  \in \Pp  (\Yy ^ m\times \Yy^n)$$ 
	with respect to the projection $\Pi_{\Yy ^n} : \Yy ^m \times \Yy^n \to \Yy^n$. Hence,  if $\qb ^n_m: \Yy^n  \to \Pp (\Yy^m)$ is a regular  conditional probability measure  for $\mu ^0_{T_m,S_n(X_m Y_n), \mu_\Theta}$ with respect  to the  projection $\Pi_{\Yy^n}$ then
	  $\qb  ^n_m \big(Y_n\big)$   coincides     with  $\Pp_{T_m|S_n (X_n, Y_n),\mu_\Theta} \in \Pp (\Yy^m)$ in  weak sense    up  to a $(\underline{\pb_{X_n}})_*\mu_\Theta$-zero set.
	
	2)   The  posterior predictive distribution $\Pp_{T_m|S_n(X_n, Y_n),\mu_\Theta} \in \Pp (\Yy^m)$  can be  computed  recursively  as  follows. 
	\begin{enumerate}\label{eq:predrecursive}
		\item    Step 1:   Let     $\qb  ^1_{m+n-1}: \Yy \to  \Pp  (\Yy  ^{m +n -1})$ be a  regular  conditional  probability measure  for 
		the joint  distribution  $\mu ^0_{T_m,S_n(X_n, Y_n), \mu_\Theta)} \in  \Pp (\Yy ^{m +n})$
		with respect to the projection  $\Yy ^{m +n } \to \Yy$.  Then we set
		\begin{equation}\label{eq:predrecursive1}
			\mu _{ (T_m,  S_n(X_n, Y_n), \mu_\Theta)}   ^ 1: =\qb   ^1_{m+n-1}  (y_n)\in \Pp (\Yy^{m +n -1}).
		\end{equation}

		\item    Step $k+1$  for $1\le k \le  n-1$.  
		Let    $\qb  ^{k+1}: \Yy \to  \Pp  (\Yy  ^{m +n -k-1})$ be a  regular  conditional  probability measure  for 
		the joint  distribution   $ \mu ^{k} _{ T_m, S_n, \mu_\Theta} \in \Pp  (\Yy  ^{m+n -k})$ with respect to the projection $\Yy  ^{m+n -k} \to \Yy$ defined inductively in the proof below. Then   we set
		\begin{equation}\label{eq:predrecursivek}
			\mu _{ (T_m,  S_n(X_n, Y_n), \mu_\Theta)}   ^{k+1}: =\qb   ^{k+1} (y_{n-k})\in \Pp (\Yy^{m +n -k-1}).
		\end{equation}

	\end{enumerate}
	
	Then    $	\mu_{ T_m, S_n(X_n, Y_n),  \mu_\Theta} ^ {n}\in \Pp  (\Yy  ^m)$ is  the  posterior predictive  distribution  of   $\Pp_{T_m | S_n(X_n, Y_n), \mu_\Theta}$.
\end{theorem}

\begin{proof}
1)   The first assertion  of Theorem  \ref{thm:posspred} is a direct  consequence  of  Proposition  \ref{prop:decompev}.

2)  For $1\le k \le  n$  we let $S_k : = \big((x_1, y_1), \ldots, (x_k, y_k)\big)$.  To prove the second assertion of  Theorem \ref{thm:posspred}, we consider the following  diagram

$$
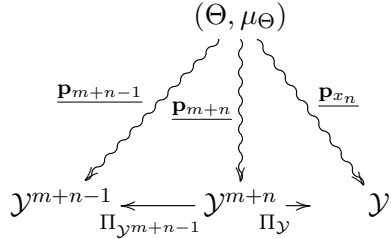
\xymatrix{
	  &  \ar@{~>}[ldd]_{\underline{\pb_{m+n-1}}} (\Theta, \mu_\Theta) \ar@{~>}[dd]_{\underline{\pb_{m+n}} } \ar@{~>}[rdd]^{\underline {\pb_{x_n}} } &\\
	  &   &	\\
	  \Yy  ^{m+n -1}  &  \ar[l]^{\Pi_{\Yy^{m+n -1}}} \Yy  ^{m +n}\ar[r ]_{\Pi _\Yy}	  &  \Yy
}
$$
where 
$$\pb_{m+n-1}: = \mathfrak  m^2  (\pb _{T_m}, \pb_{\Pi_\Xx (S_{  n-1})}), \, \pb _{m +n} : = \mathfrak  m^2  (\pb _{T_m}, \pb_{\Pi_\Xx (S_{  n})}).$$

By   Proposition \ref{prop:decompev}  we have
$$	\mu_{ (T_m, S_n,  \mu_\Theta)} ^1  = \Pp_{(T_m, \Pi_\Xx (S_{n-1}))|  (x_n, y_n),\mu_\Theta}.$$
  Next, we consider  the following   diagram 
  
  $$
  
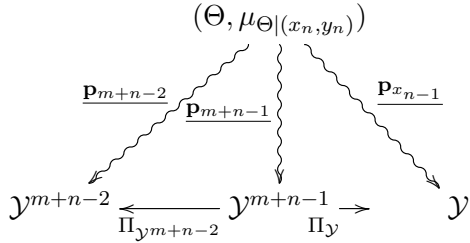
\xymatrix{
  	&  \ar@{~>}[ldd]_{\underline{\pb_{m+n-2}}} (\Theta, \mu_{\Theta| (x_n , y_n)}) \ar@{~>}[dd]_{\underline{\pb_{m+n-1}} } \ar@{~>}[rdd]^{\underline {\pb_{x_{n-1}}} } &\\
  	&   &	\\
  	\Yy  ^{m+n -2}  &  \ar[l]^{\Pi_{\Yy^{m+n -2}}} \Yy  ^{m +n-1}\ar[r ]_{\Pi _\Yy}	  &  \Yy
  }
  $$
  where 
  $$\pb_{m+n-2}: = \mathfrak  m^2  (\pb _{T_m}, \pb_{\Pi_\Xx (S_{  n-2})}).$$
  Applying Proposition \ref{prop:decompev}, taking into account Theorem \ref{thm:binv}, we obtain
 $$	\mu_{( T_m, S_n,  \mu_\Theta)} ^2  = \Pp_{(T_m, \Pi_\Xx(S_{n-2}))|  (x_n, y_n), (x_{n-1}, y_{ n-1}), \mu_\Theta}.$$ 
 Repeating this   procedure, we      obtain
 $$ \mu ^{n}_{ (T_m, S_n , \mu _\Theta)} = \Pp_{T_m| S_n, \mu_\Theta}.$$ 
\end{proof}

For  $X_m: = (x_1, \ldots, x_m) \in \Xx^m$ we denote by $[X_m]$  the   smallest   subset  of $\Xx$ that contains each  of  $x_i$. 

\begin{corollary}\label{cor:psspred}      Let  $\big(\Pp (\Yy)  ^\Xx, \mu, \Id_{\Pp (\Yy)^\Xx}, \Pp (\Yy)  ^\Xx\big)$ be a universal    Bayesian  supervised  models.   Let $T_m \in  \Xx^m, X_n \in \Xx^n$  and  $A: = [T_m] \cup [X_n]$.
	Let $R_A^\Xx  : \Pp (\Yy)^ \Xx \to \Pp (\Yy) ^A, h \mapsto  h_{|A},$  denote the restriction map.  Denote by  $\Id_{X_n }: \Pp (\Yy)^\Xx \to \Pp  (\Yy^n)$  the composition $ \mathfrak m ^n \circ E_{X_n}$. Then   for $(\underline{\Id_{X_n}})_* \mu$-a.s.   $Y_n \in \Yy^n$ we  have
	\begin{equation}\label{eq:res}
		\Pp _{ T_m| S_n(X_, Y_n), \mu} = 	\Pp_{T_m|S_n(X_n, Y_n),  (R_A^\Xx)_*  \mu}
	\end{equation}
in the weak sense, 	where the RHS of   \eqref{eq:res}  is the posterior  predictive  distribution  of the  restricted     Bayesian  supervised learning model  $(\Pp (\Yy) ^A,  (R_A^\Xx )_*\mu_\Theta,\\\Id_{\Pp(\Yy) ^A}, \Pp (\Yy)  ^A)$.
\end{corollary}	
\begin{proof}
	 We consider  the following diagram
	 \begin{equation}\label{eq:res1}
	 	\xymatrix{
	 		& \big(\Pp (\Yy)^\Xx, \mu\big)\ar[d]^{R_A^\Xx}\ar@{~>}[ddl]_{\underline{\Id_{T_m}}}\ar@{~>}[ddr]^{\underline{\Id_{X_n }}} &\\
	 		& \ar@{~>}[dl]_{\underline{\Id^A_{T_m}}}\Pp (\Yy) ^A\ar@{~>}[d]\ar@{~>}[dr]^{\underline{\Id^A_{X_n}} }  & \\
	 		\Yy^m   &  \ar[l] ^{\Pi_{\Yy^m}} \Yy^{m+n} \ar[r]_{\Pi_{\Yy^n}}  & \Yy^n
	 	}
	 \end{equation}
	where  $\Id_{T_m}:=  \mathfrak m ^m \circ E_{T_m}:  \Pp (\Yy) ^\Xx \to \Pp (\Yy^m)$  and $\Id^A_{T_m}:=  \mathfrak m ^m \circ E_{T_m} :  \Pp (\Yy) ^A \to \Pp (\Yy^m)$. Similarly, we define    $\Id^A_{X_n}$.
	By Theorem \ref{thm:posspred},  $	\Pp _{ T_m| S_n(X_n, \cdot ), \mu}: \Yy ^n \to \Pp (\Yy ^m)  $  is a regular  conditional  probability measure  for the  joint  distribution $(\underline{\Id_{T_m, X_n}})_* \mu \in \Pp (\Yy^{m+n})$, and $\Pp_{T_m|S_n(X_n, \cdot), (R_A^\Xx)_*  \mu}: \Yy^n \to \Pp (\Yy^m)$ is   a regular conditional  probability measure  for the  joint  distribution $(\underline{\Id^A_{T_m, X_n}})_*\\  (R_A^\Xx)_*\mu \in \Pp (\Yy^{m+n})$.  To conclude  Corollary \ref{cor:psspred}, we note that 
	$$(\underline{\Id_{T_m, X_n}})_* \mu  = (\underline{\Id^A_{T_m, X_n}})_*  (R_A^\Xx)_*\mu$$
	since  $[T_m] \subset A$ and  $[X_n] \subset A$, hence  $\Id_{T_m, X_n} = \Id^A_{T_m, X_n} \circ R_A^\Xx$.   Therefore,     both the  Bayesian statistical supervised model $(\Pp (\Yy^\Xx, \mu,\Id_{\Pp(\Yy) ^\Xx} )$  and the  Bayesian statistical supervised model $(\Pp (\Yy^A,  (R^\Xx_A) _* \mu,\Id_{\Pp(\Yy) ^A} )$ induce  the same joint distribution  on $\Pp (\Yy^{m+n})$.  Taking into  account   Theorem  \ref{thm:posspred},    we obtain Corollary \ref{cor:psspred}.  
\end{proof}

\begin{remark}\label{rem:DDNS}
	The preceding construction may also be relevant for Bayesian inverse
	problems with function-space unknowns. In the Gaussian-prior setting,
	such problems are often reduced, via a parametric representation of the
	prior, to posterior expectations over a countable product Gaussian
	measure; see, for example, the framework of Dinh Dung, Nguyen, Schwab,
	and Zech \cite{DNSZ2023} for PDEs with Gaussian random field inputs.
	Our formulation is different in nature: it gives a general
	measure-theoretic construction of posterior predictive laws for Bayesian
	supervised learning over Souslin spaces.  Thus it may provide a nonparametric Bayesian layer for inverse or
	regression-type problems in situations where the unknown object is
	naturally a conditional law or probability kernel, rather than a single
	Gaussian-parametric field.
\end{remark}

Let us now consider   Bayesian  regression learning, which is a particular  case  of Bayesian  supervised   learning \cite[Definition 3.10]{Le2025}.    Let  $\Xx$ be an input space and $V$ is a separable  Hilbert space. 
We consider  a  corrupted   measurement 
		\begin{equation}\label{eq:corrupted}  y   = f(x) + \eps  (x) \in V, \qquad  f \in  (V^\Xx, \mu),\qquad \eps (x)  \in  (V, \nu_\eps (x))
		\end{equation}
		where  $\nu _\eps (x)  \in \Pp (V)$  for all $ x\in   \Xx$.
		We regard   $V^\Xx$  as    a   universal parameter  space  in the  Bayesian  supervised  leanring model  $(V^\Xx, \mu,  \pb   ^\eps , \Pp  (V)^\Xx)$   where
	\begin{align} \pb ^\eps  (f) :  =\delta_f * \nu _\eps , \nonumber\\
		  \delta_f * \nu _\eps  (x) : = \delta_{f(x)} * \nu _\eps (x)\label{eq:feps}
		\end{align}
		 for learning the corrupted measurements \eqref{eq:corrupted}.
		
		In the general case,   we consider     a quadruple  $(\Theta, \mu_\Theta, h, V^\Xx)$ where  $(\Theta, \mu_\Theta)$ is a  parameter  space with a  prior   probability measure $\mu_\Theta$,   and $h \in \Meas (\Theta, V^\Xx)$.  For $X_n = (x_1,\ldots , x_n )\in \Xx^n$,  the Markov kernel   $\mathfrak m ^n  \circ E_{X_n} \circ \pb ^ \eps \circ h :  \Theta  \to  \Pp  ( V^n)$      describes the sampling distribution  of the joint distribution  of $(y_1, \ldots,  y_n)$   where
		$y_i =  f(x_i)  +\eps (x_i)$. 
		
		Let $$\pb ^0 : V ^\Xx \to  \Pp (V^\Xx), f\mapsto \delta _f, $$
		 be  the Markov  kernel
		describing the sampling distribution of  uncorrupted  measurement. By Proposition \ref{prop:decompev}, for $T_m = (t_1, \ldots,   t_m) \in \Xx^m$,  the  predictive  distribution    of the    tuple   $(f(t_1),\ldots,  f(t_m)) $    after  seeing  $S_n \in  (\Xx \times \Yy)^n$ can be chosen as   the value $\qb ^n_m  (\Pi_\Yy (S_n)) \in \Pp  ( V^m)$  where $\qb  ^n_m : V^m \to V^n$ is a regular  conditional probability measure  for  the joint distribution   $(\underline{\mathfrak  m ^2 (\mathfrak m^m \circ E_{T_m} \circ \pb ^0\circ h, \mathfrak m^n \circ E_{\Pi_\Xx (S_n)} \circ \pb ^\eps \circ h)})_* \mu_\Theta\in \Pp  (V^m \times V^n
		)$.  
		
		   We  shall     abbreviate $\mathfrak m^m \circ E_{T_m} \circ \pb ^0\circ h$
		   as   $h^0 _{T_m}$, and  $\mathfrak m^m \circ E_{T_m} \circ \pb ^\eps\circ h$
		   as $h^\eps _{T_m}$.
		 The following theorem   for   Bayesian  regression  learning  is proved  in the same way     as    Theorem \ref{thm:posspred}, so we omit its proof.
		 
		 	\begin{theorem}[Posterior  predictive distribution with corrupted measurements]\label{thm:posspredr}
		 		\
		 		
		 		  Let   $ (\Theta, \mu_\Theta, h,  V^\Xx)$   be a   Bayesian model  for regression  learning,  with $y_i$  being  a corrupted   measurement   of $  f(x_i)$ for $i \in \overline{1, n}$,   and $T_m  =  (t_1, \ldots,  t_m)\in \Xx^m$.  For  $X_n =  (x_1, \ldots, x_n) \in \Xx^n$  and  $Y_n = (y_1, \ldots, y_n)\in V^n$  let $ S_n(X_n, Y_n) : = \big ((x_1, y_1), \ldots,  (x_n, y_n)\big) \in (\Xx\times \Yy)^n$.
		 	
		 	1)  The posterior  predictive  distribution
		 	$$\Pp^n_m:  V  ^n \to \Pp (V_m), \,  Y_n  \mapsto \Pp_{ T_m| S_n (X_n, Y_n), \mu_\Theta },  $$
		 	is a  regular  conditional probability measure  for the joint distribution   
		 	$$\mu ^0_{(T_m,S_n(X_n, Y_n), \mu_\Theta)}: =\underline{\mathfrak  m^2 (h^0_{T_m}, h^\eps_{\Pi _\Xx  (S_n)})}_* \mu_\Theta  \in \Pp  (V ^ m\times V^n)$$
		 	with respect to the projection  $\Pi_{ V^n} : V^m \times  V^n \to V^n$,
		 	 Hence, if  $\qb   ^n_m : V ^n   \to \Pp (V ^m)$  is a   regular  conditional probability measure  for the joint distribution $\mu ^0_{(T_m,S_n(X_n, Y_n), \mu_\Theta)}$ 	with respect to the projection  $\Pi_{ V^n}$
		 	then  $\qb  ^n_m \big(\Pi  _V  (S_n)\big)$ coincides  with    $\Pp_{T_m|S_n(X_n, Y_n),\mu_\Theta} \in \Pp (V^m)$ in the weak sense  up to a $(\underline {h^\eps_{X_n}})_*\mu_\Theta$-zero set.

		 	2)  The posterior  predictive distribution $\Pp_{T_m|S_n(X_n, Y_n),\mu_\Theta} \in \Pp (V^m)$ can be  computed  recursively  as  follows. 
		 	\begin{enumerate}\label{eq:predrecursiver}
		 		\item    Step 1:   Let     $\qb  ^1: V \to  \Pp  (V  ^{m +n -1})$ be a  regular  conditional  probability measure  for 
		 		the joint  distribution  $\mu ^0_{(T_m,S_n(X_n, Y_n), \mu_\Theta)} \in  \Pp (V ^{m +n})$
		 		with respect to the projection  $V ^{m +n } \to V$.  Then we set
		 		\begin{equation}\label{eq:predrecursive1r}
		 			\mu _{ (T_m,  S_n(X_n, Y_n), \mu_\Theta)}   ^ 1: =\qb   ^1  (y_n)\in \Pp (V^{m +n -1}).
	\end{equation}

		 		\item   Step $k+1$  for $1\le k \le  n-1$.  
		 		Let    $\qb  ^{k+1}: V \to  \Pp  (V  ^{m +n -k-1})$ be a  regular  conditional  probability measure  for 
		 		the joint  distribution   $ \mu ^{k} _{ (T_m, S_n(X_n, Y_n), \mu_\Theta)} \in \Pp  (V  ^{m+n -k})$ with respect  to the projection  $V^{m+n-k} \to V$. Then   we set
		 		\begin{equation}\label{eq:predrecursivekr}
		 			\mu _{ (T_m,  S_n(X_n, Y_n), \mu_\Theta)}   ^ {k+1}: =\qb   ^{k+1}  (y_{n-k})\in \Pp (V^{m +n -k-1}).
		 		\end{equation}
		 		
		 	\end{enumerate}
		 	
		 	Then    $	\mu_{(T_m, S_n(X_n, Y_n),  \mu_\Theta)} ^ {n}\in \Pp  (V ^m)$ is  the posterior predictive  distribution  $\Pp_{T_m|S_n(X_n, Y_n),\mu_\Theta} \in \Pp (V^m)$.
		 \end{theorem}

			\begin{example}[Gaussian process regression]\label{ex:Gauss}   We  illustrate    Theorem \ref{thm:posspredr}  with Gaussian process regression   model   $(\R ^\Xx, \Gg\Pp (m, K), \pb ^\eps,  \Pp (\R)^\Xx)$, where    $\Gg \Pp (m, k)$  is a  Gaussian measure  on $R^\Xx$ defined  by  its mean function $m \in \R^\Xx$ and its covariance  function $K: \Xx \times \Xx \to \R$, which is a  positive  definite kernel.      Let $\mu$ be  a Gaussian measure  on the  function  space $R^\Xx$  where $\Xx$ is  an input  space and $V = \R$. One  sees immediately that    the recipe  for computing    the  posterior  predictive distribution $\Pp_{ T_m|  S_n, \Gg \Pp (m, K)}$ in  Theorem  \ref{thm:posspredr}(1) coincides  with the classical  formula  for  posterior   predictive  distributions   in  Gaussian  process  regression    described in  \cite{RW2006}.  Furthermore,  the   recursive  formula
				 in Theorem \ref{thm:posspredr}(2)  is much simpler  and faster  than   the classical  formula  since  it    does not require  computing  the (pseudo) inverse  of a  square matrix of  size $ (n\times n)$  associated  with the kernel     $K^\eps_n: \R ^n \times \R^n \to \R$  which is the variance of the Gaussian  measure that governs  the  distribution  of  $ y_i = f(x_i) +\eps(x_i) \in  \R$, $i \in \overline{1, n}$,  and a  multiplication of matrices of size $(m\times  n) $  with this  pseudo inverse and with   a matrix   of  size $(n\times m)$  \cite[Appendix A]{Stein1999},  \cite[Chapter 8]{Rao2002}. Instead,  we have to compute        $n$-round of multiplications  of  matrices  of   size    $m +n -i \times 1$ with a matrix  of   size $  1\times m +n -i$ for  $i \in \overline {1, n}$.  This sequential update procedure is known to be equivalent to the celebrated Kalman filter update equations,  see, e.g.,  \cite[\S 6.3]{SS2023}.				 

	\end{example}

	\section{Probability measures on $\Pp (\Yy)^\Xx$}\label{sec:bsuperv}

	In this Section,  we     assume that  $\Yy$ is  a Souslin metrizable space unless otherwise stated.    Recall  that  $\Pp (\Yy)^k$  and $\Pp (\Pp (\Yy)^ k)$  are Souslin  metrizable spaces for any $k \in \N^+$. We shall  extend  Orbanz's  description  of   the space  $\Pp ^2 (\Yy)$  for a Polish space  $\Yy$ \cite[Theorem 1.1]{Orbanz2011},     to   a description  of  the  space $\Pp (\Pp (\Yy)^\Xx)$, where $\Yy$ is    Souslin metrizable  space and $\Xx$   is an arbitrary index  set  using a projective system (Theorem \ref{thm:uniprior}). 
	
	For a   set $\Xx$ we denote by $\mathrm{P_{fin}}(\Xx)$ the  directed set  of finite subsets of $\Xx$.  Our projective  system is a product  of two projective  systems.
	The first  projective  system is  associated  to   the restriction maps 
	$$R_{X_m}^{X_n}:\Pp(\Yy)^{X_n} \to \Pp (\Yy) ^{X_m} \text{ if }  X_m \le X_n  \in \mathrm{P_{fin}}(\Xx).$$
	Recall that $R^\Xx_{X_m}$ denotes the  restriction map $\Pp (\Yy) ^\Xx \to \Pp (\Yy) ^{\Xx_m}$.

	\begin{lemma}\label{lem:exist1}
		Let $\Yy$ be a  Souslin space.  Then for   any   set $\Xx$ and  $\mu \in \Pp  ( \Pp (\Yy)^\Xx)$
		we have
	\begin{equation}\label{eq:exist1}
			\mu = \lim_{\stackrel{\leftarrow}{X_m \in  \mathrm{P_{fin}}(\Xx)}}  (R^\Xx_{X_m} )_* \mu.
\end{equation}
		Conversely,   for any projective  system  of probability spaces  $\{ (\Pp (\Yy)^{X_m}, \mu_{X_m}),  R_{X_m}^{X_n} : X_m  \le X_n\in \mathrm{P_{fin}}(\Xx)\}$  there  exists  a unique  probability measure
		$\mu \in \Pp \big (\Pp (\Yy)^\Xx \big)$ such that   for all $X_m \in  \mathrm{P_{fin}}(\Xx)$ we have
		$$\mu_{X_m} =   (R^\Xx_{X_m} )_* \mu .$$
	\end{lemma}
	\begin{proof}   The first assertion is straightforward.  The second assertion   is obtained by  applying the Kolmogorov  extension theorem  \cite[Corollary 7.7.2, p. 96]{Bogachev2007}.
	\end{proof}

	Next we shall   study  another projective  system  associated  with  partitions of a Souslin space $\Yy$. 
	
	 Let  $\Aa_\Yy$ be  the countable  algebra  generated by open balls with   rational radius  centered  at
	a countable dense  set  in $\Yy$.  Then  $\Aa_\Yy$  generates  the Borel  $\sigma$-algebra of  $\Yy$.   Set
	$$ \Hh (\Aa_\Yy):  = \{  (A): = (A_1, \ldots, A_n):  A_i \in \Aa_\Yy, \, \dot \cup A_i = \Yy\} . $$
	
	Note  that $\big(\Hh (\Aa_\Yy), \le \big)$  is a  directed   subset of $\big (\pi (\Yy), \le \big)$.

	\begin{remark}\label{rem:orbanz} 	
		Using \eqref{eq:compose1}  one  observes  that the collection
		$$ \{\Pp (\Om_{(A)}),  (\pi ^{(A)}_{(B)})_*:\Pp (\Om_{(A)}) \to \Pp (\Om_{(B)}),  \, (B)\le (A) \in \Hh (\Aa_\Yy)\}$$ forms  a projective system  of topological spaces.
		\end{remark}
	
For a measurable   space $\Yy$  and  $k \in \N^+$,
denote by  
$$\widehat ev_\Yy ^{(k)}: \Pp  \big (\Pp (\Yy)^k\big) \to \Pp (\Yy^k) $$ the {\it marginalization map}, 
\begin{equation}\label{eq:evS} 
	\widehat{ev}_\Yy^{(k)}( \nu) (B_1 \times \ldots \times B_k): = \int_{\Pp(\Yy)^k}\mathfrak m ^k \mu (B_1 \times \ldots \times B_k) \,  d\nu (\mu)
\end{equation}
for   $\nu  \in \Pp\big (\Pp(\Yy)^k\big)$ and $B_i \in \Sigma_\Yy$, $i \in \overline{1, k}$. 
Similarly, for a  finite set  $X_m : = \{ x_1, \ldots, x_m\}$   we denote by
$$\widehat ev_\Yy ^{X_m}: \Pp  \big (\Pp (\Yy)^{X_m}\big) \to \Pp (\Yy^{X_m})$$ the {\it marginalization map}, 
\begin{equation}\label{eq:evSm} 
	\widehat{ev}_\Yy^{X_m}( \nu) (B_1 \times \ldots \times B_m): = \int_{\Pp(\Yy)^{X_m}}\mathfrak m ^m\mu (B_1 \times \ldots \times B_m) \,  d\nu (\mu)
\end{equation}
for   $\nu  \in \Pp\big (\Pp(\Yy)^{X_m}\big)$ and $B_i \in \Sigma_{\Yy ^{\{ x_i\}}}$, $i \in \overline{1, m}$. 

	\begin{lemma}\label{lem:unim}  
			1)   Recall  that $\mathfrak  m ^k : \Pp (\Yy)^k \to \Pp (\Yy^k) $ is the  multiplication map.   Then we have
		\begin{equation}\label{eq:evev}
			\widehat{ev}^{(k)}_\Yy = P_*  \underline{\mathfrak m ^k}.
		\end{equation}
		2) The map  $\widehat{ev}_\Yy^{(k)}$ is measurable.

		3) Assume that $\Yy$ is a  Souslin space.  Given a   set  $\Xx$    and  $\nu \in \Pp  (\Pp (\Yy) ^\Xx)$  the following formula for  the marginalization $\widehat{ev}_\Yy ^\Xx (\nu)$ of $\nu$,
		\begin{equation}\label{eq:margin}\widehat{ev}_\Yy ^\Xx (\nu): =  \lim_{\stackrel{\leftarrow}{X_m \in \mathrm{P_{fin}} (\Xx)}} \widehat{ev}_\Yy ^{X_m} \big ((R^\Xx_{X_m})_*\nu\big)  \in \Pp (\Yy ^\Xx),
		\end{equation}
		is well-defined.
	\end{lemma}
	
	\begin{proof}  (1)  Equation \eqref{eq:evev} follows   directly from \eqref{eq:evS} and Equation \eqref{eq:Mhomomorphism}.
	
		(2)  The second  assertion follows from the first one,  taking into account  the measurability of  $P_* \underline{\mathfrak m^k} $    by \cite[Theorem 2]{JLT2021}.

	3) For $X_m \le X_n \in \mathrm{P_{fin}}(\Xx)$
	let  $r^\Xx_{X_m}: \Yy^\Xx \to \Yy^{X_m}$   and $r^{X_n}_{X_m}: \Yy^{X_n} \to \Yy^{X_m}$  denote the restriction maps.  
	
	By definition of the multiplication map $\mathfrak m ^{X_n}: \Pp  (\Yy) ^{X_n} \to \Pp (\Yy ^{X_n})$,   we   verify immediately  that the following   diagram
	\begin{equation}\label{eq:proj0}
		\xymatrix{
			\Pp (\Yy)^{X_n}\ar[r]^{R^{X_n}_{X_m}}
			\ar@{~>}[d]^{\underline{\mathfrak m ^{X_n}}} & 
			\Pp (\Yy)^{X_m}\ar@{~>}[d]^{\underline{\mathfrak m ^{X_m}}} \\
			\Yy^{X_n}	\ar[r]^{r^{X_n}_{X_m}} & \Yy^{X_m}
		}
	\end{equation}
	is commutative. Consequently,  the 
	  following diagram  for any $X_m \le X_n \in \mathrm{P_{fin}} (\Xx)$:
	\begin{equation}\label{eq:proj1}
		\xymatrix{
			\Pp \big(\Pp (\Yy)^{X_n}\big)\ar[r]^{P_*(R^{X_n}_{X_m})}
			\ar[d]^{P_*\underline{\mathfrak m ^{X_n}}} & 
			\Pp \big(\Pp (\Yy)^{X_m}\big)\ar[d]^{P_*\underline{\mathfrak m ^{X_m}}} \\
			\Pp (\Yy^{X_n})	\ar[r]^{(r^{X_n}_{X_m})_*} & \Pp (\Yy^{X_m})	
		}
	\end{equation}
	is commutative.   Combining the  commutativity  of  the diagram \eqref{eq:proj1}, Equation \eqref{eq:evev},  and the second   assertion of  Lemma \ref{lem:unim}, we  complete  the  proof    of Lemma \ref{lem:unim},  taking into account  the Kolmogorov extension theorem.     	
	\end{proof}
	
	 Let $\Yy$ be a Souslin space and $\Xx$  a set.  For  $ (B) \le (A)  \in \Hh (\Aa_\Yy)$, and for $X_m \le X_n \in \mathrm{P_{fin}} (\Xx)$ we denote by
	 $$ R^{X_n} _{(A),  X_m}:  \Om_{(A)} ^{X_n} \to \Om _{(A)} ^{X_m}$$
	 the restriction map, and by
	 $$\pi ^{(A), X_m}_{ (B)}: \Om_{(A)} ^{X_m} \to \Om_{ (B)} ^{X_m}$$
	 the natural  projection map.
	 	
	\begin{theorem}\label{thm:uniprior}  Assume  that $\Yy$ is a Souslin space and
		$\Aa_\Yy$  is the  countable algebra generating  $\Bb (\Yy)$ defined above.  Then	 for  any $\nu \in \Pp (\Pp (\Yy)^\Xx)$ we have
		\begin{equation}\label{eq:uniprior}
			\nu = \lim_{\stackrel{\leftarrow}{X_m \in \mathrm{P_{fin}} (\Xx)}}  \lim_{\stackrel{\longleftarrow}{(A)\in \Hh (A_\Yy)}}  P^2_*\pi_{(A)}^{X_m}  (R^\Xx_{X_m})_*\nu
		\end{equation} 
		and 
		\begin{equation}\label{eq:limmar}
			\widehat {ev}_{\Yy}^\Xx  (\nu)  =  \lim_{\stackrel{\leftarrow}{X_m \in \mathrm{P_{fin}} (\Xx)}} \lim_{\stackrel{\longleftarrow}{(A)\in \Hh (\Aa_\Yy)}}P_*\pi_{(A)}^{X_m}  (r^\Xx_{X_m})_*\widehat {ev}_{\Yy}^\Xx  (\nu) . 
		\end{equation}
	
		Conversely,  given   a   projective  system  of finite sample  spaces  endowed  with  second order   probability  measures
		$$\Big \{ \Big(\Om^{X_m}_{(A)}, \nu_{(A)}^{X_m}\in \Pp \big(\Pp (\Om _{(A)})^{X_m}\big)\Big): X_m \in \mathrm {P_{fin}}  (\Xx), A \in	 \Hh(\Aa_\Yy)\Big\}$$
		and   induced  projection maps
		$$ \Big\{ (R_{(A), X_m}^{X_n})_*: \Pp (\Om_{(A)}) ^{X_n} \to \Pp ( \Om_{(A)} )^{X_m}, $$
		$$  P_*(R_{(A), X_m}^{X_n})_*:\Pp 
		\big( \Pp (\Om_{|A|}) ^{X_n}\big) \to \Pp \big (\Pp ( \Om_{(A)}) ^{X_m}\big),$$  
		$$ P_* ^2 \pi^{ (A), X_m}_{(B)}:\Pp  \big(\Pp (\Om_{ (A)}) ^{X_m}\big) \to \Pp \big  (\Pp (\Om_{(B)}) ^{X_m}\big)|\: $$
		$$  X_m \le X_n  \in  \mathrm{P_{fin}} (\Xx),\,   (B) \le (A)\in \Hh (\Aa_\Yy)\Big\}  $$
		there exists   $\nu \in  \Pp \big(\Pp (\Yy)^\Xx\big)$  such that  for any $ (A)\in \Hh (\Aa_\Yy)$ and $X_m \in \mathrm{P_{fin}}(\Xx)$ we have
		\begin{equation}\label{eq:limp2}
			\nu_{(A)}^{X_m} = P^2_*\pi_{(A)}^{X_m}  (R^\Xx_{X_m}) _*\nu \in \Pp (\Pp (\Om _{ (A)})^{X_m})
		\end{equation}
		if and only if  there  exists   $\mu \in \Pp (\Yy^\Xx)$ such that   for   any  $X_m \in \mathrm{P_{fin}} (\Xx) $ and any $ (A) \in \Hh (\Aa_\Yy)$ we have
		\begin{equation}\label{eq:orbanzm}\widehat{ev}_{\Om _{(A)}} ^{(m)}\big (\nu ^{X_m} _{ (A)}\big) = P_* \pi_{(A)}^{X_m}  (r^\Xx_{X_m})_*    \mu \in \Pp (\Om_{(A)}^{X_m}). 
		\end{equation}
		Equivalently, there exists   $\nu \in  \Pp \big(\Pp (\Yy)^\Xx\big)$   such that \eqref{eq:limp2}  holds,   if and only if    for   each  $X_m \in \mathrm{P_{fin}} (\Xx) $  there exists  $\mu_{X_m} \in \Pp (\Yy^{\Xx_m})$  such  that for  any $ (A) \in \Hh (\Aa_\Yy)$ we have
		\begin{equation}\label{eq:orbanzma}\widehat{ev}_{\Om _{(A)}} ^{(m)}\big (\nu ^{X_m} _{ (A)}\big) = P_* \pi_{(A)}^{X_m}  ( \mu_{X_m}) \in \Pp (\Om_{(A)}^{X_m}), 
		\end{equation} 
		 where  the   system $\{r_{X^m}^{X_n}: (\Yy ^{X_n}, \mu_{X_n}) \to  (\Yy ^{X_m},\mu_{X_m}) , \, X_m \le X_n \in \mathrm{P_{fin} (\Xx)} \}$ of probability spaces  is projective.
	\end{theorem}
	
	\begin{remark}\label{rem:uniprior}  (1)  For the case $\Xx $ consists of  one  element,  Theorem \ref{thm:uniprior} is   due to Orbanz \cite[Theorem 1.1]{Orbanz2011}.
		
		(2)  Any  $\mu \in \Pp (\Yy^\Xx)$   can be  written as
		$$\mu = \lim_{ \stackrel{\leftarrow}{X_m \in \mathrm{P_{fin}} (\Xx) }}  (r^\Xx_{X_m})_* (\mu).$$
		By the Kolmogorov  extension theorem,  we can   replace $(r^\Xx_{X_m})_* \mu$ in \eqref{eq:orbanzm}  by $\mu_{X_m}$ in \eqref{eq:orbanzma} in  the presence  of  the corresponding  projective  system, which is the content of the last  ``equivalence" assertion of Theorem \ref{thm:uniprior}.
	\end{remark}
	
	     The first  step on our  proof  of  Theorem \ref{thm:uniprior} is to  generalize  \cite[Theorem 1.1]{Orbanz2011} to the case of Souslin spaces.
	     
	 \begin{proposition}\label{prop:orbanzsouslin1}   Theorem \ref{thm:uniprior}    holds, if $\Xx = \{ pt\}$.
	 \end{proposition}
	 \begin{proof} Assume the  condition  of   Theorem \ref{thm:uniprior} with $\Xx = \{pt\}$.  Let $\mu \in \Pp (\Yy)$  satisfy the conditions of Theorem  \ref{thm:uniprior}.    Let $\overline{\Yy}$ be the completion  of $\Yy$.  Then $\overline \Yy$   and $\Pp (\overline \Yy) $ are Polish spaces.
	 	Let us choose a countable dense set in $\Yy$.  Then   this set is also dense  in $\overline  \Yy$. Thus we  can write   
	 	\begin{align}
	 		\Aa_\Yy = \big(\Aa_{ \overline \Yy}\big)_{| i (\Yy)} \label{eq:restr1}
	 	\end{align}
	 	where $ i: \Yy \to \overline \Yy$ is the continuous inclusion.  Since $i$ is a continuous inclusion, $ i_*: \Pp (\Yy) \to \Pp (\overline \Yy)$ is  a continuous  inclusion in the   weak*-topology.  Then $i_* \mu \in \Pp (\overline \Yy)$. By \cite[Theorem 1.1]{Orbanz2011}, which  is   Proposition \ref{prop:orbanzsouslin1} specialized   for  Polish  label spaces,  there  exists  $\bar \nu \in \Pp^2 (\overline  \Yy)$ such that \eqref{eq:limp2} holds  for any $(\bar A) \in \Hh (\Aa_{\overline \Yy})$. By   the assumption of Theorem \ref{thm:uniprior},  we have  
	\begin{equation}
			\widehat {ev}_{\overline {\Yy}}  (\bar \nu)   =i_*\mu.\label{eq:bary}
	\end{equation}   We claim that    $\bar \nu  \in \Pp^2 \big( i (\Yy)\big)\subset \Pp^2 (\overline  \Yy)$.
	 	
	 	\begin{lemma}\label{lem:outerv}  Set
	 		$$ C \coloneqq \overline \Yy \setminus  \Yy . $$ Then 
	 		\begin{equation}\label{eq:outerv}
	 			\lambda^*  (C) = 0  \text{ for } \overline  \nu\text{-a.e. }  \lambda \in \Pp (\overline \Yy).
	 		\end{equation}
	 	\end{lemma} 
\begin{proof} Assume the contrary. Then there exists $\varepsilon>0$ such that
\[
	 		\overline\nu
	 		\left\{
	 		\lambda\in\Pp(\overline\Yy):
	 		\lambda^*(C)>\varepsilon
	 		\right\}>0.
\]
Set
\[
	 		C(\overline\nu,\varepsilon)
	 		\coloneqq
	 		\left\{
	 		\lambda\in\Pp(\overline\Yy):
	 		\lambda^*(C)>\varepsilon
	 		\right\}.
\]
For every
	 		$\lambda\in C(\overline\nu,\varepsilon)$ and every Borel set
	 		$B\subset\overline\Yy$ satisfying $C\subset B$, we have
	 		\[
	 		\lambda(B)>\varepsilon.
	 		\]
	 		Consequently, for every Borel set $A\subset i(\Yy)$,
	 		\[
	 		C\subset\overline\Yy\setminus A,
	 		\]
	 		and hence
	 		\[
	 		\lambda(A)<1-\varepsilon.
	 		\]
	 		
	 		Since $i(\Yy)$ is analytic and $i_*\mu$ is concentrated on
	 		$i(\Yy)$, there exist increasing Borel sets
	 		\[
	 		A_n\subset i(\Yy)
	 		\]
	 		such that
	 		\[
	 		i_*\mu(A_n)\longrightarrow1.
	 		\]
	 		Using \eqref{eq:bary}
	 		\[
	 		\widehat{ev}_{\overline\Yy}(\overline\nu)=i_*\mu,
	 		\]
	 		we obtain
	 		\[
	 		\lim_{n\to\infty}
	 		\int_{\Pp(\overline\Yy)}
	 		\lambda(A_n)\,d\overline\nu(\lambda)
	 		=
	 		1.
	 		\]
	 		On the other hand,
	 		\[
	 		\begin{aligned}
	 			\int_{\Pp(\overline\Yy)}
	 			\lambda(A_n)\,d\overline\nu(\lambda)
	 			&\le
	 			(1-\varepsilon)
	 			\overline\nu\bigl(C(\overline\nu,\varepsilon)\bigr)\nonumber\\
	 			&\quad+
	 			1-
	 			\overline\nu\bigl(C(\overline\nu,\varepsilon)\bigr)\nonumber\\
	 			&=
	 			1-\varepsilon\,
	 			\overline\nu\bigl(C(\overline\nu,\varepsilon)\bigr)
	 			<1,
	 		\end{aligned}
	 		\]
	 		which is a contradiction.
	 		This  completes  the proof of Lemma \ref{lem:outerv}.
	 	\end{proof}
	 	
	 	{\it Completion of the  proof of Proposition \ref{prop:orbanzsouslin1}.}
	 	
	 	It follows from  Lemma \ref{lem:outerv}  that   for any $\lambda \in  \Pp(\overline \Yy)$ we have  $\lambda_*  (i (\Xx)) =1$  for $\overline \nu$-a.e. $\lambda$.  Therefore  $\overline  \nu \in \Pp^2  \big(i (\Yy)\big)$, what is required to prove.
	 \end{proof}    
	
	\begin{proof} [Proof of  Theorem \ref{thm:uniprior}] (1)    The   equality  \eqref{eq:uniprior}   is a consequence  of the functoriality  $P_*: \Probm  \to \Meas$ that assigns  each measurable  space $\Xx$ to  measurable space  $\Pp (\Xx)$ and
		each  probabilistic morphism  $T \in \Probm  (\Xx, \Yy)$ to a measurable  mapping  $P_* T \in \Meas (\Pp (\Xx), \Pp (\Yy))$ \cite[Theorem  1]{Giry1982}.
		
		The  equality \eqref{eq:limmar}  follows from Lemma \ref{lem:unim} and Remark \ref{rem:uniprior} (2).
		
		Now let us prove the last assertion of Theorem \ref{thm:uniprior}. The ``only if" assertion  is a consequence  of \eqref{eq:uniprior} and \eqref{eq:limmar}.
		
		Now    we assume the    ``if"  condition of Theorem \ref{thm:uniprior}. 
		For each $X_n \in \mathrm{P_{fin}} (\Xx)$, 
		   we consider the projective system  of probability  spaces
			$$\Big\{\Big ( \Pp (\Om^{X_n}_{(A)}), \mathfrak m ^n ( \nu ^{X_n}_{(A)} )\in   \Pp (\Om^{X_n}_{(A)})\Big )\Big \}$$
			together  with   mappings 
			$$P_* \pi^n _{ (A)}: \Pp (\Yy^{X_n}) \to  \Pp (\Om ^{X_n}_{ (A)}), $$
			$$ P_* ^2  \pi^n _{ (A)}: \Pp ^2(\Yy^{X_n}) \to  \Pp^2 (\Om ^{X_n}_{ (A)}). $$
			Taking into  account Proposition \ref{prop:orbanzsouslin1} for $\Pp ^2 (\Yy^{X_n})$, see Remark \ref{rem:uniprior}(1),  we conclude    that  there exists  $\tilde \nu_{X_n} \in \Pp^2 (\Yy ^{X_n})$ such that
			\begin{equation}\label{eq:limn}
				\mathfrak m ^n   (	\nu ^{X_n} _{ (A)}) = P_* ^2  \pi_{ (A)}^{X_n}(\tilde  \nu_{ X_n}).
			\end{equation}
		Now we  consider the following   commutative  system
		\begin{equation}\label{eq:comm3-u}
			\xymatrix{\Pp(\Yy) ^{X_n} \ar[r] ^{\mathfrak m ^n }\ar[d] ^{(\pi_{ (A)})_* ^{X_n}} &  (\Pp (\Yy  ^{X_n}), \tilde \nu _{X_n})\ar[d] ^{P_*  ( \pi_{ (A)}^{X_n})}\\
				\Pp (\Om _{ (A)})^{X_n}\ar[r] ^{ \mathfrak m ^n} & \Pp \big(\Om _{ (A)}^{X_n}, \mathfrak m ^n (\nu^{X_n}_{(A)}) \big).
			}
		\end{equation}	
		\begin{lemma}\label{lem:uniprior}   The  image  $\mathfrak m ^n(\Pp (\Yy )^{X_n})$  is a measurable  subset  of  $\Pp (\Yy ^{\Xx_n})$.
	\end{lemma}
\begin{proof}[Proof of Lemma \ref{lem:uniprior}]    Let   $ \Aa_\Yy$ be a countable  algebra  generating $\Sigma _\Yy$.  Then   
		$$ \mathfrak m ^n(\Pp (\Yy )^{X_n}) = \Big\{ \mu  \in    \Pp (\Yy  ^{X_n}):  \mu  (A_1 \times \ldots \times  A_n)=$$
	$$ = \Pi_{ i =1}   ^n \mu(\underbrace{\Yy \times\ldots \times  \Yy}_{(i-1)\, times}\times A_i\times \underbrace{\Yy \times \ldots \times \Yy}_{(n-i)\, times}) \text{ for  any } A_i \in \Aa_\Yy, i \in \overline{1, n}\Big\}$$ 
	is a measurable  subset  of  $\Pp (\Yy ^{\Xx_n})$.
	  This completes  the proof  of Lemma \ref{lem:uniprior}.
	 \end{proof}	 	
{\it Completion of the proof of Theorem \ref{thm:uniprior}.}	Since  the map $\mathfrak m ^n$ is injective,    taking into account    of  Lemma  \ref{lem:uniprior}, we conclude that  
$$\tilde  \nu_{X_n}  = (\mathfrak m ^n)_*  \nu_{X_n}$$
 for some $\nu_{X_n} \in \Pp (\Pp (\Yy) ^{X_n})$. 
			Noting that   the system $\{(\Pp (\Yy) ^{X_n},\nu_{X_n}), X_n \in \mathrm{P_{fin}} (\Xx)\}$ of probability spaces satisfies  the condition of  the   Kolmogorov  extension  theorem, we complete  the proof  of Theorem \ref{thm:uniprior}.
		\end{proof}

		\section{MacEachern's  Dependent  Dirichlet Processes and Bayesian supervised learning}\label{sec:BDDP}

		\subsection{MacEachern's  Dependent  Dirichlet Processes revisited}\label{subs:DDP}
		
		In this subsection,  using Theorem \ref{thm:uniprior}, we revisit MacEachern's Dependent Dirichlet Processes (DDPs) \cite{MacEachern1999} \cite{MacEachern2000} by synthesizing the categorical framework of this paper with the copula-based construction by Barrientos, Jara, and Quintana \cite{BJQ2012}.  As MacEachern  \cite{MacEachern2000} and  Barrientos, Jara, and Quintana \cite{BJQ2012}, we  assume that $\Yy$ is a measurable  subset  of $\R^n$  and $\Xx$ is an arbitrary index  set. 
		
		Let us first recall the general definition of a DDP from \cite[Definition 1]{BJQ2012}. A DDP is   generated   by a map (stochastic  process) 
		$$G: \Om \times  \Xx \to \Pp (\Yy)$$  where $(\Om, P)$ is a  probability space and    for each $x \in \Xx$ the map $G (\cdot, x):  (\Om, P) \to \Pp (\Yy)$ is measurable, or equivalently, the   map 
		$$ \hat G:  (\Om, P) \to \Pp (\Yy) ^\Xx, \hat G (\om)(x): = G(\om, x),$$
		 is measurable. Furthermore, motivated by Sethuraman's work \cite{Sethuraman1994},  $G$ must satisfy the following  condition. For any   $x \in \Xx$ and $B \in \Sigma_\Yy$  we have
		\begin{equation}
			G (\om, x) (B)= \sum_{i=1}^{\infty} W_i(\om,x) \delta_{\theta_i(\om,x)}(B), \, \text{ for   $P$-a.e. }\; \om \in \Om
		\end{equation}
		where   for   all $x\in \Xx$ and $P$-a.e.  \; $\om \in \Om$ 
		$$W_i(\om, x) = V_i(\om, x) \prod_{j<i} (1-V_j(\om,x)),$$   with  $V_i$ and $\theta_i: \Om \times \Xx \to \Yy$   described below.
		\begin{enumerate}
			\item $\{V_i:  \Om \times \Xx   \to  [0,1]\}_{i=1}^\infty$     where for  each $x\in \Xx$ the     sequence $\{V_i  (\cdot, x) : \Om   \to  [0,1]\}_{i=1}^\infty$   are  i.i.d.    such that   for any $i$ 
			$$\big (V_i (\cdot, x)\big)_* P = \text{Beta}(1, \alpha (x)) \in \Pp([0,1]) \text{ where } \alpha (x) \in \R_{>0} .$$ 
			The dependence structure of $V_i (\cdot , x)$ across $x \in \Xx$ for each $i$ is determined by a family  $\Cc_\Xx^V: =\{ C^V_{x_1, \ldots, x_d}: [0,1]^d \to [0,1]\}$ of copula  functions describing     finite dimensional   CDF of  $(\hat V_i )_* P\in \Pp ([0,1]^\Xx)$  where $\hat V_i:  \Om \to [0,1]^\Xx, \hat  V _i  (\om)(x) : = V_i (\om, x) $.
			\item $\{\theta_i:  \Om \times \Xx \to \Yy\}_{i=1}^\infty$  where    for each $x$  the sequence   $\{ \theta_i(\cdot, x):\Om \to  \Yy\}_{i =1} ^\infty$ are i.i.d.  such that for any $i$
			$$\big (\theta_i  (\cdot, x)\big)_* P =  G^0_x \in \Pp  (\Yy). $$
			 The dependence structure of $\theta _i (\cdot, x)$ across $x \in \Xx$ for each $i$ is determined by a family  $\Cc_\Xx^\theta:= \{ C^\theta_{x_1, \ldots, x_d}: [0,1]^d \to [0,1] \}$ of copula functions  describing  finite dimenional CDF of $(\hat \theta_i)_*P \in \Pp (\Yy ^\Xx)$ where $\hat \theta_i: \Om \to \Yy^\Xx$, $\hat \theta_i(\om) (x): = \theta_i (\om, x)$.
		\end{enumerate}
		 We denote the induced probability measure $(\hat  G) _* P \in \Pp (\Pp (\Yy)^\Xx)$     by $\Dd \Dd\Pp(\alpha_\Xx\in \R_{>0} ^\Xx, \Cc_\Xx^\theta, \Cc_\Xx^V, G^0_\Xx\in \Pp (\Yy)^\Xx)$.
   In fact, $(G, P)$ can   be  chosen  as $ \big((\Pp  (\Yy) ^\Xx,   \Dd\Dd \Pp (\alpha_\Xx,\\
     \Cc  ^\theta_\Xx, C^V_\Xx, G^0_\Xx)\big)$  and   $G$ is  defined to be the natural  evaluation mapping: $G (\om, x) (B)  : = \om(x) (B)$ for any $\om \in \Pp  (\Yy) ^\Xx, x\in \Xx$ and $B \in \Sigma_\Yy$.
     
		According to Theorem \ref{thm:uniprior}, the probability measure  $\Dd\Dd\Pp(\alpha_\Xx, \Cc_\Xx^\theta, \Cc_\Xx^V, G^0_\Xx)$  is uniquely determined by the projective system of its finite-dimensional projections. Let us describe this system.
		For any finite set of predictors $X_m = \{x_1, \dots, x_m\} \subset \Xx$ and any finite measurable partition $(A) = (A_1, \dots, A_k)$ of $\Yy$, the corresponding projection is the probability measure
		$$ \nu_{(A)}^{X_m} := P^2_*\pi_{(A)}^{X_m} (R^\Xx_{X_m}) _* \Dd\Dd\Pp(\alpha_\Xx, \Cc_\Xx^\theta, \Cc_\Xx^V, G^0_\Xx) \in \Pp(\Pp(\Om_k)^{X_m}). $$
		More explicitly, let
		$$ \pb_{x_i}: = \big(G  (\cdot , x_i)(A_1), \ldots, G (\cdot, x_i)(A_k)):  \Om \to  \Delta _k : = \Pp (\Om_k). $$
		Then 
		$$\nu_{(A)}^{X_m} =  (\pb_{x_1}, \ldots, \pb_{x_m})_*P \in \Pp(\Pp(\Om_k)^{X_m}).$$
		
		The structure of this  probability measure $ \nu_{(A)}^{X_m}$ is as follows:
		\begin{itemize}
			\item For any fixed $x_i \in X_m$, the marginal distribution $(\pb_{x_i})_* P\in \Pp   ^2  (\Om_k)$ is a Dirichlet distribution, as $\big(G  (\cdot, x_i)\big)_*P\in \Pp^2 (\Yy)$ is  a Dirichlet process. Specifically,
			$$ (\pb_{x_i})_*  P = \text{Dir}(\alpha (x_i) G^0 (x_i)(A_1), \dots, \alpha (x_i) G^0 (x_i)(A_k)). $$
			\item The crucial point is that the joint distribution $\nu_{(A)}^{X_m}\in \Pp(\Pp(\Om_k)^{X_m})$ is not a simple product of these marginal Dirichlet distributions  $(\pb_{x_i})_* P\in \Pp   ^2  (\Om_k)$. The dependence between $\pb_{x_i}$ and $\pb_{x_j}$ for $i \neq j$ is induced by the dependence structure of the underlying stick-breaking processes $\{V_l(x)\}_{l=1}^\infty$ and $\{\theta_l(x)\}_{l=1}^\infty$. This dependence is precisely what is encoded by the copula families $\Cc_\Xx^V$ and $\Cc_\Xx^\theta$.
		\end{itemize}
		Furthermore, Theorem \ref{thm:uniprior} provides a consistency condition involving a ``mean" measure $\mu \in \Pp(\Yy^\Xx)$. For the  $\Dd\Dd\Pp$, this corresponds to the map of base measures $G^0_\Xx: \Xx \to \Pp(\Yy)$ defined by $x \mapsto G^0_x$, regarded  as an element in $\Pp (\Yy) ^\Xx \stackrel{\mathfrak m ^\Xx}{\INTO} \Pp (\Yy ^\Xx)$. The $\Dd\Dd\Pp$ is centered around this collection of measures, as $\E_P[G_x] = G^0_x$. The projective system $\{\nu_{(A)}^{X_m}, X_m \in (\mathrm{P_{fin}} (\Xx), \le)\}$ must satisfy the condition \eqref{eq:orbanzm}:
		$$ \widehat{ev}_{\Om_k}^{(m)}(\nu_{(A)}^{X_m}) = P_*\pi^{X_m}_{(A)}\circ  (r^\Xx_{X_m})_* (\mathfrak m ^\Xx G^0_\Xx). $$
		This simply states that the expected value of the random vector $\pb_{x_i}$ is the vector of probabilities of the base measure, $(G^0_{x_i}(A_1), \dots, G^0_{x_i}(A_k))$, which is a fundamental property of the Dirichlet process \cite[\S 4.1.4]{GV2017}.
		
		In summary, we can characterize MacEachern's $\Dd\Dd\Pp$ in the following way:
		\begin{theorem}\label{thm:ddp}
			The law of a Dependent Dirichlet Process, $\Dd\Dd\Pp(\alpha_\Xx, \Cc_\Xx^\theta, \Cc_\Xx^V, G^0_\Xx)$, is the unique probability measure $\nu \in \Pp(\Pp(\Yy)^\Xx)$ that satisfies the two conditions of Theorem \ref{thm:uniprior}, where:
			\begin{enumerate}
				\item The projective system of second-order probability measures $\{\nu_{(A)}^{X_m}\}$ is defined such that each $\nu_{(A)}^{X_m}$ is the law of an $m$-tuple of  measurable  mappings $\{\pb_{ x_i}: \Om \to \Pp (\Om_k)\}_{ i=1} ^m$, where the marginal law  $ (\pb_{ x_i})_* P  \in \Pp^2 (\Om_k)$ is Dirichlet, and their joint dependence structure is determined by the copula families $\Cc_\Xx^V$ and $\Cc_\Xx^\theta$.
				\item The corresponding projective system of first-order measures is given by the base measure map $\mu =\mathfrak m ^\Xx \circ G^0_\Xx \in \Pp (\Yy ^\Xx)$.
			\end{enumerate}
		\end{theorem}
		 The same reasoning can be applied to characterize the simpler ``single-weights" and ``single-atoms" $\Dd\Dd\Pp$s from \cite[Definitions 2, 3]{BJQ2012} by appropriately simplifying the copula structure (e.g., using independence copulas for the weights in the single-weights case).
		
\subsection[Posterior predictive distributions for DDP priors]{Posterior  predictive  distributions  of  Bayesian supervised learning models $ \big(\Pp(\Yy) ^\Xx, \Dd\Dd\Pp, \Id_{\Pp (\Yy) ^\Xx}, \big(\Pp(\Yy) ^\Xx \big)$}

  Let  us consider  a Bayesian  supervised  learning  model 	 $ \big(\Pp(\Yy) ^\Xx, \Dd\Dd\Pp, \Id_{\Pp (\Yy) ^\Xx},
   \Pp(\Yy) ^\Xx )$  where  $\Dd\Dd\Pp = \Dd\Dd\Pp (\alpha_\Xx, \Cc^\theta_\Xx, \Cc^V_\Xx,  G^0_\Xx)$  is described  in  Subsection \ref{subs:DDP}. In particular, $\Yy $ is a measurable subset in $\R^k$.  To compute 
    the posterior  predictive distribution $\Pp_{T_m|S_n, \Dd\Dd\Pp}$   we restrict $\Dd\Dd\Pp $ to $\Pp (\Yy) ^ A$ where 
  $ A = [T_m] \cup  [\Pi_\Xx (S_n)] $.   The  restriction  $(R^\Xx_{A})_*\Dd\Dd\Pp$  of $\Dd\Dd\Pp$ to $\Pp (\Yy)^A$  is      $\Dd\Dd \Pp(\alpha_A,  \Cc ^\theta_A,  \Cc^V _A,  G^0_A)$ where  
    $$ \alpha _A  =  (\alpha_\Xx)_{ | A},  \Cc^\theta_A = (\Cc ^\theta_\Xx)_{| A}, 
    \Cc^V_A = (\Cc ^V_\Xx)_{| A}, G^0_A= (G^0_\Xx)_{|A}.$$
 Next, we shall apply Theorem \ref{thm:posspred} to compute
 $\Pp_{T_m\mid S_n,\Dd\Dd\Pp}$, where
 $T_m=(t_1,\ldots,t_m)$ and
 $\Pi_\Xx(S_n)=(x_1,\ldots,x_n)$. Put
 \[
 A=[T_m]\cup[\Pi_\Xx(S_n)].
 \]
 Then
 \[
 \mu^0_{T_m,S_n,\Dd\Dd\Pp}
 =
 \int_{\Pp(\Yy)^A}
 \left(
 \bigotimes_{i=1}^m h(t_i)
 \otimes
 \bigotimes_{j=1}^n h(x_j)
 \right)
 d(R^\Xx_A)_*\Dd\Dd\Pp(h).
 \]
 Thus
 \[
 \mu^0_{T_m,S_n,\Dd\Dd\Pp}\in \Pp(\Yy^{m+n}).
 \]
 If $\Yy\subset \R^k$, then
 $\mu^0_{T_m,S_n,\Dd\Dd\Pp}$ may be regarded as an element of
 $\Pp((\R^k)^{m+n})$.
 
 The one-coordinate marginals of $\mu^0_{T_m,S_n,\Dd\Dd\Pp}$ are given by
 \[
 \mu^0_{t_i}
 =
 \int_{\Pp(\Yy)^A} h(t_i)\,
 d(R^\Xx_A)_*\Dd\Dd\Pp(h),
 \qquad i=1,\ldots,m,
 \]
 and
 \[
 \mu^0_{x_j}
 =
 \int_{\Pp(\Yy)^A} h(x_j)\,
 d(R^\Xx_A)_*\Dd\Dd\Pp(h),
 \qquad j=1,\ldots,n.
 \]
 Consequently, to determine the joint measure
 $\mu^0_{T_m,S_n,\Dd\Dd\Pp}$ it remains to determine the corresponding
 copula
 \[
 C_{T_m,S_n,(R^\Xx_A)_*\Dd\Dd\Pp}.
 \]
 Knowing this copula and the above marginals, we can apply
 Theorem 3.4.1 in \cite{DS2016} to compute
 $\Pp_{T_m\mid S_n,\Dd\Dd\Pp}$, using the recursive formula in
 Theorem \ref{thm:posspred}.

 Alternatively, to compute  the  posterior  predictive   distribution $\Dd\Dd\Pp_{ T_m| S_n, \Dd\Dd\Pp}$, we may apply   Theorem \ref{thm:uniprior}  and    Theorem \ref{thm:projbi}, or the following Theorem. 
 
 \begin{theorem}\label{thm:pospredlim}  Let $\Yy$ be a  measurable  space,  $\Xx: = \{ x_1, \ldots, x_n\}$  a finite set, $S_n \in (\Xx\times \Yy)^n$,   and  $X_n = \Pi_{ \Xx}(S_n) \in \Xx^n$,  $T_m \in \Xx ^m$. Assume that for any $(A)\in \pi (\Yy)$  there  exists a   Markov kernel $\qb ^n_{  (A), m}: \Om_{(A)}^n \to \Pp (\Om_{(A)}^m)$  which is a   regular  conditional  probability measure of the joint  distribution  of  $\pb^{(A)}_{T_m, X_n} ( P_*  (\pi_{(A)})_* ^\Xx \mu)$  such that 
 the following diagram is commutative for any $(A) \le (B) \in \pi (\Yy)$.		
 	$$ 
 	\xymatrix{\Pp (\Om_{(B)}^m)  \ar[d]_{ P_*(\pi^{(B)}_{ (A)})^m} & & \ar[ll]_{\qb_{(B), m}^{n}} \Om_{ (B)} ^n  \ar[d]^{ ( \pi^{(B)}_{ (A)})^n}\\
 		 \Pp (\Om_{(A)}^m) &  &\ar[ll]_{\qb_{(A), m}^{n}}\Om_{ (A)} ^n 	.
 	}
 	$$
 	Assume  that there  exists  a  Markov kernel $\qb^n_m : \Yy^n \to \Pp (\Yy^m)$ such that   for any  $ (A) \in \pi (\Yy)$ the following  diagram is commutative
 	$$ 
 	\xymatrix{\Pp (\Yy^m )\ar[d]_{ P_*(\pi_{ (A)})^m} & & \ar[ll]_{\qb^n_m} \Yy ^n  \ar[d]^{(\pi_{ (A)})^n}\\
 		\Pp (\Om_{(A)}^m) & & \ar[ll]_{\qb ^n_{(A),m }}\Om_{ (A)} ^n	.
 	}
 	$$
 	Then    $\qb ^n_m (\Pi_\Yy  (S_n))= \Pp_{ T_m|S_n, \mu}$.
 	\end{theorem}
 	
 	This Theorem  is proved in the same  way  as    Theorem \ref{thm:projbi}, so we omit its proof.
 s
\section{Final remarks}\label{sec:fin}
\begin{enumerate}
	\item In this  paper  we   proved  that batch  Bayesian learning  equals  Bayesian online learning  under the assumption of conditionally independent   data,  making  Bayesian learning more    efficient   in the presence  of complex  data.    While the sequential nature of Bayesian updating has long been recognized - most explicitly in the Kalman filter for linear Gaussian models and in conjugate exponential families -previous formulations relied on the existence of probability densities (dominated models) or specific algebraic structures (conjugate priors). Theorem \ref{thm:binv} shows that the equivalence of batch and online learning follows fundamentally from the categorical structure of probabilistic morphisms and conditional independence, without requiring these additional assumptions.  
\item  Bayesian regression learning  with corrupted measurements can be extended  to   nonlinear spaces $\Yy$ where  we can   model    measurement  error using probability  measures, e.g., for   homogeneous  Riemannian  manifolds $\Yy$.    Corollary \ref{cor:psspred} can be extended  for    Bayesian  regression learning  with  corrupted measurements  in the same way.

\item    It     is important  to     find a   suitable  concept   of  predictive  consistency  of  Bayesian  supervised  learning which would agrees  with the classical concept     and   the concept in a recent work  by
 P. Koerpernik  and F. Pfaff.  A possible solution is to  introduce the notion of predictive consistency   at  a finite subset  $A \subset \Xx$  and consider    posterior  predictive distributions of the form  $\Pp_{T_m| S_n, \mu}$  where $[T_m] \subset A$   and  $[\Pi_\Xx (S_n)]\in A$, moreover, $\Pi_\Xx (S_n)$ visits  each   element of $A$ infinitely many times.  If $\# \Xx = 1$   and  the  sampling  operator  is Markov kernel   this  concept  is the    notion  of posterior   consistency in  classical  Bayesian  statistics \cite[\S 6.8.3]{GV2017}.  Furthermore, this notion also  agrees  with  the   condition  of  recurrent  density  entered in the  concept  of   posterior  consistency of Gaussian process  regressions  in      \cite{KP2021}. 
\end{enumerate}

\section*{Acknowledgement}
		Research of HVL was supported by the Institute of Mathematics,  Czech Academy of Sciences (RVO: 67985840).  The author  would like to thank Steven  MacEachern for   suggesting the paper by  Barrientos-Jara-Quintana to  her.  She  also thanks Xuan Long Nguyen for inviting her to the workshop  ``Bayesian Modeling, Computation,
		and Applications"  at Ho Chi Minh City in July 2025, which  stimulated  her  working  over this  paper.

	\end{document}